\documentclass{article}
\usepackage[affil-it]{authblk}
\usepackage{amsthm, graphicx, latexsym}  
\usepackage[utopia]{mathdesign}      
 \usepackage[mathscr]{eucal}
\usepackage{amsmath, yhmath}  
\usepackage{array, longtable, sidecap, wrapfig, subfig, color}
\usepackage[normalem]{ulem}
\usepackage{cancel}

\usepackage[colorlinks=true, pdfstartview=FitV, linkcolor=blue, citecolor=blue, urlcolor=cerulean]{hyperref}
\usepackage{fancyhdr}
\begin{document}

\newtheorem{thm}{Theorem}[section]
\newtheorem*{thm*}{Theorem}
\newtheorem{cor}[thm]{Corollary}
\newtheorem{lemma}[thm]{Lemma}
\newtheorem{lemmadef}[thm]{Lemma/Definition}
\newtheorem{proposition}[thm]{Proposition}
\newtheorem{prop}[thm]{Proposition}

\newenvironment{definition}[1][Definition]{\begin{trivlist}
\item[\hskip \labelsep {\bfseries #1}]}{\end{trivlist}}
\newenvironment{example}[1][Example]{\begin{trivlist}
\item[\hskip \labelsep {\bfseries #1}]}{\end{trivlist}}
\newenvironment{remark}[1][Remark]{\begin{trivlist}
\item[\hskip \labelsep {\bfseries #1}]}{\end{trivlist}}
\newenvironment{remarks}[1][Remarks]{\begin{trivlist}
\item[\hskip \labelsep {\bfseries #1}]}{\end{trivlist}}

\theoremstyle{definition}
\newtheorem{defn}[thm]{Definition}
\newtheorem{defns}[thm]{Definitions}
\newtheorem{rem}[thm]{Remark}
\newtheorem{Rem}[thm]{Remark}
\newtheorem{Rems}[thm]{Remarks}
\newtheorem{Exercise}[thm]{Exercise}
\newtheorem{Example}[thm]{Example}
\newtheorem{Examples}[thm]{Examples}
\newtheorem{Open questions}[thm]{Open questions}
\newtheorem{Open question}[thm]{Open question}
\newtheorem{Open problems}[thm]{Open problems}
\newtheorem{Open problem}[thm]{Open problem}

\newcommand{\leftexp}[2]{{\vphantom{#2}}^{#1}\!{#2}}
\newcommand{\beq}{\begin{eqnarray*}}
\newcommand{\eeq}{\end{eqnarray*}}
\newcommand{\ol}{\overline}
\newcommand{\Supp}{\textrm{Supp}}
\newcommand{\Span}{\textrm{Span}}
\newcommand{\Gr}{\textbf{Gr} }
\newcommand{\GL}{\textrm{GL}}
\newcommand{\SL}{\textrm{SL}}
\newcommand{\Matr}{\textbf{Matr}}
\newcommand{\mat}[9]{\left(\begin{array}{ccc}#1&#2&#3\\#4&#5&#6\\#7&#8&#9\end{array}\right)}
\newcommand{\rowm}[3]{\left(\begin{array}{c}#1\\#2\\#3\end{array}\right)}
\newcommand{\rk}{\textrm{rk}~}
\newcommand{\im}{\textrm{im}~}
\newcommand{\Div}{\textrm{div}~}
\newcommand{\z}{\overline}
\newcommand{\Par}[2]{\frac{\partial#1}{\partial#2}}
\newcommand{\hf}{\frac{1}{2}}
\newcommand{\lf}{\left(}
\newcommand{\rt}{\right)}
\newcommand{\re}{\textrm{Re}}
\newcommand{\img}{\textrm{Im}}
\newcommand{\limn}{{\lim\atop{n\to\infty}}}
\newcommand{\Res}{\textrm{Res}}
\newcommand{\eps}{\epsilon}
\newcommand{\tp}{\frac{1}{2\pi i}}
\newcommand{\of}{\circ}
\newcommand{\ub}{\underbrace}
\newcommand{\K}{\tilde{k}}
\newcommand{\floor}[1]{{\lfloor#1\rfloor}}
\newcommand{\Gal}{\textrm{Gal}}
\newcommand{\Aut}{\textrm{Aut}}
\newcommand{\out}[1]{{\textrm{Out}(F_#1)}}
\newcommand{\pr}[2]{\langle#1,#2\rangle}
\newcommand{\Tr}{\textrm{Tr}}
\newcommand{\id}{\mathfrak}
\newcommand{\rad}{\mathfrak{R}}
\newcommand{\Spec}{\mathrm{Spec}}
\newcommand{\tensor}{\otimes}
\newcommand{\Ann}{\textrm{Ann}}
\renewcommand{\phi}{\varphi}
\newcommand{\Ker}{\textrm{Ker}~}
\newcommand{\Img}{\textrm{Im}~}
\newcommand{\Z}{\mathbb{Z}}
\newcommand{\N}{\mathbb{N}_{\geq0}}
\newcommand{\trans}[1]{{\leftexp{T}{#1}}}
\def\onto{{\kern3pt\to\kern-8pt\to\kern3pt}}
\newcommand{\IO}{\textrm{IO}}
\newcommand{\IA}{\textrm{IA}}
\newcommand{\Stab}{\textrm{Stab}}
\newcommand{\ssim}[1]{\stackrel{(#1)}{\sim}}
\newcommand{\Lk}[1]{\textrm{Lk}_<(#1)}
\newcommand{\set}[1]{\left\{#1\right\}}
\newcommand{\ssm}{\smallsetminus}
\newcommand{\abs}[1]{\left|#1\right|}
\newcommand{\Dist}{\textup{Dist}}

\newcommand{\wh}{\widehat}
\newcommand{\wt}{\widetilde}
\newcommand{\wb}{\overline}
\newcommand{\ihat}{\hat{\imath}}
\newcommand{\fhat}{\hat{f}}
\newcommand{\ibar}{\overline{\imath}}
\newcommand{\fbar}{\overline{f}}

\def\ms{\medskip}
\renewcommand{\ss}{\smallskip}
\newcommand{\bs}{\bigskip}

\newcommand{\tc}[2]{\textcolor{#1}{#2}}
\definecolor{cerulean}{rgb}{0,.48,.65} \newcommand{\cerulean}[1]{\tc{cerulean}{#1}}
\definecolor{magenta}{rgb}{.5,0,.5} \newcommand{\magenta}[1]{\tc{magenta}{#1}}
\definecolor{dred}{rgb}{.5,0,0} \newcommand{\dred}[1]{\tc{dred}{#1}}
\definecolor{green}{rgb}{0,.5,0} \newcommand{\green}[1]{\tc{green}{#1}}
\definecolor{blue}{rgb}{0,0,0.5} \newcommand{\blue}[1]{\tc{blue}{#1}}
\definecolor{black}{rgb}{0,0,0} \newcommand{\black}[1]{\tc{black}{#1}}
\definecolor{dgreen}{rgb}{0,.3,0} \newcommand{\dgreen}[1]{\tc{dgreen}{#1}}
\definecolor{vdred}{rgb}{.3,0,0} \newcommand{\vdred}[1]{\tc{vdred}{#1}}
\definecolor{red}{rgb}{1,0,0} \newcommand{\red}[1]{\tc{red}{#1}}
\definecolor{salmon}{rgb}{0.98,0.50,0.45} \newcommand{\salmon}[1]{\tc{salmon}{#1}}
\definecolor{gray}{rgb}{.5,.5,.5} \newcommand{\gray}[1]{\tc{gray}{#1}}
\definecolor{seagreen}{rgb}{0.13,0.70,0.67} \newcommand{\seagreen}[1]{\tc{seagreen}{#1}}
\definecolor{chartreuse}{rgb}{0.40,0.80,0.00}\newcommand{\chartreuse}[1]{\textcolor{chartreuse}{#1}}
\definecolor{cornflower}{rgb}{0.39,0.58,0.93} \newcommand{\cornflower}[1]{\textcolor{cornflower}{#1}}
\definecolor{gold}{rgb}{0.80,0.68,0.00}\newcommand{\gold}[1]{\textcolor{gold}{#1}}

\setlength{\parindent}{0pt}
\setlength{\parskip}{7pt}

\title{Cannon--Thurston maps, subgroup distortion, \\ and hyperbolic hydra}

\author{Owen Baker%
 \thanks{\texttt{otb2@cornell.edu}}}
\affil{Department of Mathematics, 900 University Avenue, \\  University of California, Riverside, CA 92521, USA}

\author{Timothy Riley%
\thanks{\texttt{tim.riley@math.cornell.edu}. Partial  support from NSF grant DMS--1101651 and from Simons Foundation Collaboration Grant 208567 is gratefully acknowledged.}}
\affil{Department of Mathematics, 310 Malott Hall, \\  Cornell University, Ithaca, NY 14853, USA}

\date \today

\maketitle

\begin{abstract}       
\noindent   There is a family of hyperbolic groups known as \emph{hyperbolic hydra}  which contain heavily distorted free subgroups.  We prove the existence of Cannon--Thurston maps (that is, maps of the boundaries induced by subgroup inclusion) for these free subgroups.  
It is  known that Cannon--Thurston maps between hyperbolic \emph{space} boundaries can exist even in the presence of arbitrarily heavy (even non-recursive) distortion.
The hyperbolic hydra examples show that Cannon--Thurston maps can exist even between hyperbolic \emph{group} boundaries in the presence of arbitrarily heavy primitive recursive distortion.

\ms \noindent  \textbf{2010 Mathematics Subject Classification:  20F67 }   
\end{abstract}

\section{Introduction} \label{intro} 
An isometry of hyperbolic $n$-space $\mathbb{H}^n$ induces a homeomorphism on $\partial \mathbb{H}^n=S^{n-1}$, the sphere at infinity.  More generally, an isometric embedding $\mathbb{H}^m\hookrightarrow\mathbb{H}^n$ induces an embedding $S^{m-1}\hookrightarrow S^{n-1}$; a quasi-isometric embedding $X\to Y$ of (Gromov-)hyperbolic spaces induces an embedding $\partial X\hookrightarrow\partial Y$ of the Gromov boundaries.   Remarkably, many natural embeddings which are far from being isometric nonetheless induce maps on the boundaries.   These are known as \emph{Cannon--Thurston maps} in honor of  the researchers who gave the first exotic example: let $M$ be a hyperbolic 3-manifold fibering over the circle with hyperbolic surface fiber $S$ and pseudo-Anosov monodromy.  The inclusion $S\hookrightarrow M$ of a fiber induces an embedding $f:\mathbb{H}^2=\widetilde{S}\hookrightarrow\widetilde{M}=\mathbb{H}^3$ of the universal covers, which they showed induces a (surjective!) map $\fhat:S^1\to S^2$ \cite{Cannon-Thurston}.

Cannon \& Thurston's example can be viewed as a map $\partial\pi_1S=S^1\to S^2=\partial\mathbb{H}^3$ induced by an orbit map $\pi_1S\to\mathbb{H}^3$ of the surface Kleinian group $\pi_1S$.  Thus one way of generalizing the Cannon--Thurston example involves replacing $\pi_1S$ with an arbitrary finitely generated Kleinian group.  In a series of papers  culminating in \cite{Mj}, Mahan Mj (formerly  Mitra) showed that for any finitely generated Kleinian group $G$, the orbit map $G\to\mathbb{H}^3$ always extends continuously to the the boundary, inducing a \emph{Cannon--Thurston map} $\partial G\to S^2$.

Another interpretation of Cannon \& Thurston's example is as a map $\partial\pi_1S=S^1\to S^2=\partial\pi_1 M$ induced by the inclusion of (Gromov)-hyperbolic groups $\pi_1S\hookrightarrow\pi_1M$.  (Section~\ref{hyperbolic review} contains background on hyperbolic groups and their boundaries.)  This leads to another direction of generalization:  Let $\Lambda\leq\Gamma$ be hyperbolic groups.
If the inclusion map $f:\Lambda\hookrightarrow\Gamma$ extends to a continuous map $\fbar:\Lambda\cup\partial\Lambda\to\Gamma\cup\partial\Gamma$ of the Gromov compactifications, then $\fbar$ (or its restriction $\fhat:\partial\Lambda\to\partial\Gamma$) is called a \emph{Cannon--Thurston map.} 
When a Cannon--Thurston map exists, it is unique.

Mj showed this map exists when  $\Lambda$ is an infinite  hyperbolic normal subgroup of  a hyperbolic group  $\Gamma$ \cite{CTnormal}.    He also showed  it exists  when   $\Gamma$ is a hyperbolic group which is a   finite graph of hyperbolic  groups, with $\Lambda$ one of the (infinite) vertex- or edge-groups,  under the assumption that the edge inclusions are quasi-isometric embeddings \cite{CTtrees}.    We gave the first example  of a hyperbolic group with hyperbolic subgroup, for which there is no  Cannon--Thurston map \cite{BR2}.   Matsuda \& Oguni   showed our example  leads to examples where the subgroup in question can be any non-elementary hyperbolic group or, even, relatively hyperbolic group  \cite{MO}.

The fact that the original Cannon--Thurston map $\fhat:S^1\to S^2$ of \cite{Cannon-Thurston} is surjective (space-filling) stems from the difference between the intrinsic metric of the hyperbolic surface fiber $S$ and the ambient metric from the hyperbolic 3-manifold $M$.  That is, \emph{from the fact that} $\widetilde{S}\hookrightarrow\widetilde{M}$ (or $\pi_1S\hookrightarrow\pi_1M$) is \emph{distorted}.

For a finitely generated subgroup $\Lambda$ of a finitely generated group $\Gamma$,  define the \emph{distortion function} $$\Dist^{\Gamma}_{\Lambda}(n)  \ : = \  \max \set{ \, d_{\Lambda}(e,h) \, \mid \, h \in {\Lambda}, \, d_{\Gamma}(e,h) \leq n \, },$$ where $d_{\Gamma}$ and $d_{\Lambda}$ are word metrics with respect to some finite generating sets.  We say that $f \preceq g$ for  $f,g: \mathbb{N} \to \mathbb{N}$  when there exists $C>0$ such that $f(n) \leq Cg(  Cn+C ) +Cn+C$ for all $n \geq 0$.  We say $f  \simeq g$ when $f \preceq g$ and $g \preceq f$.   Up to $\simeq$, $\Dist^{\Gamma}_{\Lambda}(n)$ does not depend on the choices of finite generating sets.   A similar definition  applies in the Kleinian groups context: $$\Dist^{\mathbb{H}^3}_{G}(n) \  : = \  \max \set{ \, d_{G}(e,g) \, \mid \, g \in G, \, d_{\mathbb{H}^3}(x_0,g\cdot x_0) \leq n \, }.$$

If $\Lambda$ is an undistorted subgroup in a hyperbolic group $\Gamma$ (that is, $\Dist_{\Lambda}^{\Gamma} (n)  \preceq n$), then $\Lambda$ is also hyperbolic (e.g.\ \cite[page~461]{BH1}) and the  Cannon--Thurston map $\partial{\Lambda} \to \partial{\Gamma}$ is readily seen to   exist and be injective. 

It is natural to ask whether extreme distortion is an obstacle to the existence of a Cannon--Thurston map.  In the Kleinian group setting, it is not.    Cannon--Thurston maps   always exist for surface Kleinian groups:  McMullen \cite{McMullen} proved this in the punctured torus case and Mj~\cite{Mj} in the general case.  
And Mj proved that distortion functions can be arbitrarily wild in this setting \cite[p.160--161]{CTtrees}, even non-recursive, by a construction based on ideas of  Minsky from \cite{MinskyPuncturedTorus}.

For hyperbolic subgroups of hyperbolic groups, the relation (or lack thereof) between distortion and the existence of Cannon--Thurston maps is less clear.  As we mentioned earlier, Mj proved they exist for infinite normal hyperbolic subgroups of hyperbolic groups.  These are never more than exponentially distorted. Further, Mj  established in  \cite[p.159--160]{CTtrees} that  they exist  for all $k \geq 1$  for certain examples  (based on a construction of Bestvina, Feighn, \& Handel~\cite{BFHLaminations}) which display  $k$-fold iterated exponential distortion.   There are examples of heavier distortion where the existence of the  Cannon--Thurston map remains unknown, specifically   Mj's   example where the distortion exceeds a $k$-fold iterated exponential for all $k$ (but by Corollary~3 of \cite{Bernasconi} (unpublished) is dominated by the function $A_4$ discussed below), and the closely related CAT(-1) examples of Barnard, Brady \& Dani~\cite{BBD}. 
 
In this article  we show that  arbitrarily fast-growing primitive recursive  distortion is no barrier to the  existence of Cannon--Thurston maps for hyperbolic subgroups of hyperbolic groups.

The  \emph{hyperbolic hydra} $\Gamma_k$ ($k=1, 2, \ldots$) of  \cite{HypHydra} are a family of hyperbolic groups  with finite-rank free subgroups  $\Lambda_k$ exhibiting the fastest-growing distortion functions  known for hyperbolic subgroups of hyperbolic groups:     $\Dist^{\Gamma_k}_{\Lambda_k}$ grows at least like $A_k$,  the $k$-th of  Ackermann's family of increasingly fast-growing functions  that begins with $A_1(n)=2n$,  $A_2(n)=2^n$, and $A_3(n)$ equalling the height-$n$ tower of powers of $2$.   
 Any primitive recursive function is dominated by some $A_k$ (see for example pages~11--21 of \cite{Calude}).   We will prove in Section~\ref{CT for hydra}:

\begin{thm} \label{existence thm}
Hyperbolic hydra have Cannon--Thurston maps $\partial{\Lambda_k}\to\partial{\Gamma_k}$ for all $k$.
\end{thm}

While  heavy distortion may fail to obstruct the existence of Cannon--Thurston maps, it is natural to expect it to manifest as some sort of wildness in the map: witness, for example, how  Cannon \& Thurston's original example      is a space-filling curve $S^1 \onto S^2$ \cite{Cannon-Thurston}.  More generally, wildness manifests in Cannon--Thurston maps  in the relationship between ``$\varepsilon$''  and ``$\delta$'' in their continuity (having given the boundaries natural metrics).  
This was made precise by Miyachi~\cite{Miyachi}  in the Kleinian group setting when the group is a finitely generated Fuchsian group of the first kind with bounded geometry and no parabolic elements:
Miyachi \cite{Miyachi} gives an upper bound on the modulus of continuity for the Cannon--Thurston map, and shows it is not H\"older continuous when the group is geometrically infinite.  

In this article we will establish the corresponding result for hyperbolic subgroups of hyperbolic groups.  
Here are more details.  The \emph{modulus of continuity} $\varepsilon: [0, \infty) \to [0, \infty]$  of a function $f:U\to V$ between metric spaces  is 
$$\varepsilon(\delta)  \ := \ \sup \{  d_V(f(a), f(b))  \mid  a,b \in U \textup{ with }  d_U(a,b) \leq \delta \}.$$
This notion goes back at least to Lebesgue in 1909 \cite{Lebesgue}.  
An upper bound on $\varepsilon(\delta)$ expresses a degree of good behaviour:  $f$ is  uniformly continuous if $\varepsilon(\delta)\to0$ as $\delta\to0$;   is Lipschitz  
if $\varepsilon(\delta)\leq C\delta$ for a constant $C >0$; and is  $\alpha$-H\"older if $\varepsilon(\delta) \leq C\delta^\alpha$  for a constant $C >0$.   
Wildness manifests in lower bounds on $\varepsilon(\delta)$, expressing that $\varepsilon(\delta)$ is extravagantly larger than $\delta$ when $\delta$ is close to $0$.   We will prove:

\begin{thm}
\label{Distortion and Wildness}
Suppose $\Lambda \leq \Gamma$ are hyperbolic, $\Lambda$ is non-elementary,  the Cannon--Thurston map $\partial \Lambda \to \partial \Gamma$  exists, and $r, s >1$ are any    visual parameters for visual metrics on $\partial \Lambda$ and $\partial \Gamma$, respectively.   Then there exist $\alpha,\beta>0$ so that for all $n \geq 0$ the modulus of continuity satisfies 
\[ \varepsilon\left(\frac{\beta}{r^{\Dist_{\Lambda}^{\Gamma}(n)}}\right) \ \geq \ \frac{\alpha}{s^n}. \]
\end{thm}

While we have not been able to find this theorem elsewhere in the literature, we understand that it is known to some experts.  Indeed, the second author first heard of a suggestion  of a relationship between the modulus of continuity of Cannon--Thurston map and subgroup distortion from Mahan Mj.   The moduli of continuity of Cannon--Thurston maps have received attention before in the Kleinian group setting.  For a faithful discrete representation to $\textup{PSL}_2(\mathbb{C})$ of a  finitely generated Fuchsian group of the first type with bounded geometry and no parabolic elements,  Miyachi \cite{Miyachi} gives an upper bound on the modulus of continuity of the Cannon--Thurston map.

The reason we insist  that $\Lambda$ be non-elementary  (i.e.\  contains an $F_2$ subgroup; see e.g. \cite[Theorem~2.28]{KB})  in this theorem is that elementary $\Lambda$ are not interesting in this context: if $\Lambda$ is finite, then $\partial \Lambda$ is empty; if $\Lambda$ is virtually $\Z$, then $\Lambda$ is quasiconvex in $G$, its boundary $\partial \Lambda$ is two points, and the Cannon--Thurston map   exists and is an embedding.

Theorem~1.1 of \cite{HypHydra} says that $\Dist^{\Gamma_k}_{\Lambda_k} \succeq A_k$.  It combines with
Theorems~\ref{existence thm} and \ref{Distortion and Wildness}   to give  that the Cannon--Thurston map $\partial{\Lambda_k}\to\partial{\Gamma_k}$ has the property that for  $\delta$ incredibly small, $\varepsilon(\delta)$ is, by comparison, huge.   More precisely, we will prove in  Section~\ref{Wildness}:

\begin{cor}  \label{wildness thm}
Fix $k \geq 3$ and any visual metrics on $\partial{\Lambda_k}$ and $\partial{\Gamma_k}$.  The modulus of continuity for the Cannon--Thurston  map $\partial{\Lambda_k}\to\partial{\Gamma_k}$ for hyperbolic hydra satisfies
$$\varepsilon\left(\frac{1}{A_{k-1}(n)}\right)  \ \geq \  \frac{1}{n}$$ for all sufficiently large $n$.
In particular, for any primitive recursive function $f$ there exists a Cannon--Thurston map \emph{between hyperbolic groups} satisfying
$$\varepsilon\left(\frac{1}{f(n)}\right) \ \geq \ \frac{1}{n}$$ for all sufficiently large $n$. 
\end{cor}

(The final part of the corollary follows immediately,  since every primitive recursive function is dominated by some $A_k$, as we remarked earlier, citing \cite{Calude}.)

A detailed understanding of the Cannon--Thurston Map $\partial{\Lambda_k}\to\partial{\Gamma_k}$ appears hard to obtain.  Whilst  $\partial \Lambda_k$ is a Cantor set (as $\Lambda_k$ is free),  $\partial \Gamma_k$  is not so readily identified.  
(I.~Kapovich \& M.~Lustig \cite{KL} recently made advances in the understanding of Cannon--Thurston maps for certain free-by-cyclic groups, but  $\Lambda_k \leq \Gamma_k$ do not fall within the scope of their work.)
Here is what we can say about $\partial \Gamma_k$.

Splittings of hyperbolic free-by-cyclic groups $F\rtimes_\phi\Z$ are studied in \cite{KK} and \cite{Brink}, the former dealing with the case where $\phi$ is an irreducible hyperbolic free group automorphism, and the latter with $\phi$ a general hyperbolic free group automorphism. The argument  preceding Corollary~15 in \cite{KK} shows that any hyperbolic free-by-cyclic group has a one-dimensional boundary: the cohomological dimension of any (finitely generated free)-by-cyclic group is 2 (see e.g. \cite[pp.185--7]{Brown}), so \cite[Corollary 1.4(d)]{BM} implies $\partial \Gamma_k$ has dimension $2-1=1$. 

The argument of \cite[Corollary~15]{KK} shows that any hyperbolic free-by-cyclic group has connected, locally connected boundary without global cut points.  To see this, it suffices by \cite[Theorems~7.1 and 7.2]{KB} to check that $F\rtimes_\phi\Z$ is freely indecomposable, which is true for \emph{any} free group automorphism $\phi$. Indeed, the Bass--Serre tree $T$ for any graph of groups decomposition of $F\rtimes_\phi\Z$ admits a minimal action by the normal subgroup $F$ with quotient a finite graph by Grushko's Theorem.  This shows the edge stabilizers for the action of $F\rtimes_\phi\Z$ on $T$ are non-trivial, so the decomposition cannot be free. 

On the other hand, $\Gamma_k$ splits as an HNN-extension over $\Z$ for every $k$, so \cite[Theorem~7.2]{KB} implies $\partial\Gamma_k$ has local cut points.  Indeed, $\Gamma_1$ splits over $\Z$ as $\langle B, a_1 \rangle*_{\langle a_1^{-1} t^2ua_1=t^2v^{-1} \rangle}$,  where $B$ denotes the subgroup generated by all the defining generators other than  $a_0$ and $a_1$ (then $a_0$ appears as $t^{-1}a_1t$), and for $k \geq 2$,  $\Gamma_k$ splits as an  HNN-extension over $\Z$ with the stable letter $a_k$ conjugating $t$ to $t a^{-1}_{k-1}$.    
 
For additional background on Cannon--Thurston maps we recommend Mj's recent survey \cite{MjSurvey}.

\emph{The organization of this article.}  In Section~\ref{hyperbolic review} we give background on hyperbolic groups and their boundaries.  In Section~\ref{Mitra's lemma section} we define Cannon--Thurston maps  and prove an embellished version of a lemma of Mitra giving necessary and sufficient conditions for their existence.  In Section~\ref{CT for hydra} we review the construction of hyperbolic hydra groups and  prove Theorem~\ref{existence thm}.  In Section~\ref{Wildness} we prove Theorem~\ref{Distortion and Wildness} and Corollary~\ref{wildness thm}. 

\emph{Acknowledgment.}  We thank Mahan Mj  for conversations which fueled an interest in the relationship between subgroup distortion and Cannon--Thurston maps.  We are grateful to an anonymous referee for a careful reading and for valuable guidance on the Kleinian-groups literature.

\section{Hyperbolic groups and their boundaries} \label{hyperbolic review}

This section contains a brief account of some pertinent background.  More general treatments can be found in, for example, \cite{BH1, GhysdelaHarpe, KB, Short} and Gromov's foundational article \cite{Gromov4}.  

For a metric space $X$, the \emph{Gromov product} $(a \cdot b)_e$ (or $(a \cdot b)^X_e$ if there is  ambiguity) of $a,b \in X$ with respect to  $e \in X$   is  $$(a \cdot b)_e \ = \  \frac{1}{2}( d(a,e) + d(b,e) - d(a,b)).$$ 

One says $X$ is \emph{$(\delta)$-hyperbolic} when $$(a \cdot b)_e  \ \geq \  \min \set{(a \cdot c)_e, (b \cdot c)_e } - \delta$$ for all $e,a,b,c \in X$, and $X$ is hyperbolic when it is $(\delta)$-hyperbolic for some $\delta \geq 0$.   When $X$ is a geodesic space this is equivalent to  other standard definitions of  hyperbolicity (such as  $\delta$-thin or $\delta$-slim triangles),  although the $\delta$ involved may not agree.    

When $X$ is $(0)$-hyperbolic and geodesic---that is, an $\mathbb{R}$-tree---$(a \cdot b)_e$ is the distance from $e$ to the geodesic  between $a$ and $b$.   Correspondingly, in a $(\delta)$-hyperbolic geodesic space every pair of geodesics  from $e$ to $a$ and to  $b$, both parametrized by arc-length,  $6 \delta$-fellow-travel for approximately  $(a \cdot b)_e$ and  then diverge (by the \emph{insize}  characterization of hyperbolicity of \cite[page~408]{BH1}).    Indeed:
\begin{lemma} \label{6 delta}
In a $(\delta)$-hyperbolic geodesic metric space,  for every geodesic $[a,b]$ connecting $a$ and $b$
$$| d(e, [a,b]) - (a \cdot b)_e |  \ \leq  \ 6 \delta.$$   
\end{lemma}

\begin{proof}
See \cite{BH1}: the proof of Proposition~1.22 on page 411 shows that \emph{insizes}  of geodesic triangles are at most $6\delta$, and  the proof of Proposition~1.17(3)$\implies$(2) on page 409 shows that all geodesic triangles are \emph{$6\delta$-thin}, and the claimed inequality follows.
\end{proof}

The  \emph{(Gromov-) boundary} $\partial X$ of a hyperbolic metric space $X$ is defined with reference to, but is in fact independent of, a  point $e \in X$.  It is the set of equivalence classes of sequences $(a_n)$ in $X$ such that  $(a_m \cdot a_n)_e \to  \infty$ as $m,n \to \infty$, where two such sequences  $(a_n)$ and $(b_n)$ are equivalent when  $(a_m \cdot b_n)_e \to  \infty$ as $m,n \to \infty$.  Indeed, they are equivalent when  $(a_n \cdot b_n)_e \to  \infty$ as $n \to \infty$ since $$(a_n\cdot b_m)_e  \ \geq \ \min\{(a_n\cdot b_n)_e,(b_m\cdot b_n)_e\} - \delta$$ by $(\delta)$-hyperbolicity.  
Denote the equivalence class of  $(a_n)$ by $\lim a_n$.  

 When $X$ is a \emph{geodesic} hyperbolic metric space, there are equivalent definitions of $\partial X$, such as   $\partial X$  is the set of equivalence classes of geodesic rays emanating from $x$, where two such rays are equivalent when they stay uniformly close.   
 The condition $(a_m \cdot a_n)_e \to  \infty$ is what makes a  sequence $(a_n)$ ray-like, and the condition $(a_m \cdot b_n)_e \to  \infty$ corresponds to uniform closeness.  
  
Extend the Gromov product to $\wb{X} : = X \cup \partial X$ by 
$$(a  \cdot b )_e \ = \  \sup \liminf_{ m,n \to \infty} (a_m \cdot b_n)_e$$ 
where the $\sup$ is over all sequences $(a_m)$ and $(b_n)$ in $X$ representing (when in $\partial X$) or tending to (when in $X$) $a$ and $b$, respectively.    (The  ``$\sup \liminf$''  is  necessary---see   \cite[page 432]{BH1}.)   

We note, for \ref{3} in the following lemma, that in a proper geodesic hyperbolic metric space $X$, each pair of distinct points  $a,b \in \partial X$ is joined by a bi-infinite geodesic line $[a,b]$ (Lemma~3.2 on page 428 of \cite{BH1}).

\begin{lemma}
\label{BH1 433 Lemma}
Suppose $X$ is a proper  geodesic  $(\delta)$-hyperbolic metric space.
\vspace*{-6pt}
\begin{enumerate}
\renewcommand{\theenumi}{\textup{(\arabic{enumi})}}
\addtocounter{enumi}{-1}
\item If $x,y\in\overline{X}$ and $e\in X$, then there exist sequences $(x_n)$ and $(y_n)$ in $X$ with $x=\lim x_n$, $y=\lim y_n$, and $(x\cdot y)_e=\lim_n(x_n\cdot y_n)_e$. \label{0}
\item \label{BH1 433 Lemma 1} If $a,b,c\in\ol{X}$ and $e\in X$, then $(a\cdot b)_e \ \geq \ \min\{(a\cdot c)_e,(c\cdot b)_e\}-2\delta$.    \label{1}
\item \label{BH1 433 Lemma 2} If $a,b\in X$ and $c\in\partial X$, then $|d(a,b)-(a\cdot c)_b-(b\cdot c)_a| \ \leq \ \delta$.   \label{2}
\item \label{BH1 433 Lemma 3} If $e \in X$ and $[a,b]$ is any geodesic  joining any $a,b\in\partial X$, then $|d(e,[a,b])-(a\cdot b)_e| \ \leq \  8\delta$.   \label{3}
\end{enumerate}
\end{lemma}
\begin{proof}  \ref{0}  and  \ref{1} are parts  3  and  4  of \cite[page 433, Remark 3.17]{BH1}.   (Alternatively, see parts 3 and 5 of  \cite[Lemma~4.6]{Short} but note that there $\inf \liminf$ is used in place of $\sup \liminf$ and so the constants differ.) 

For \ref{2}, using \ref{0} take sequences $c_n,c_n'$ both approaching $c$ such that $(a\cdot c)_b=\lim(a\cdot c_n)_b$ and $(b\cdot c)_a=\lim(b\cdot c_n')_a$.  Now, $(\delta)$-hyperbolicity yields $(a\cdot c_n)_b\geq\min\{(a\cdot c_n')_b,(c_n\cdot c_n')_b\}-\delta$. As $n\to\infty$, we have $(c_n\cdot c_n')_b\to\infty$, but $(a\cdot c_n')_b$ is bounded above by $d(a,b)$.  So   $(a\cdot c_n')_b\leq(a\cdot c_n)_b+\delta$ for all sufficiently large $n$. Interchanging the roles of $c_n$ and $c_n'$ we find $|(a\cdot c_n)_b-(a\cdot c_n')_b|\leq\delta$ for all sufficiently large $n$.  Hence:
\[
|d(a,b)-(a\cdot c_n)_b-(b\cdot c_n')_a| \ \leq \ \delta+|d(a,b)-(a\cdot c_n')_b-(b\cdot c_n')_a| \ =  \  \delta+|0| \  = \ \delta 
\]
for all sufficiently large $n$.  Taking the limit as $n\to\infty$ gives the result.

For \ref{3} (cf.\ Exercise 3.18(3)~\cite[page~433]{BH1}),  choose sequences $a_n\to a$ and $b_n\to b$ along $[a,b]$.  Also choose $a_n',b_n'\in X$ as in (0) so that $(a\cdot b)_e=\lim(a_n'\cdot b_n')_e$.   For large enough $n$, the closest point of $[a,b]$ to $e$ lies on $[a_n,b_n]$, so   
\begin{equation}
|d(e,[a,b])-(a_n\cdot b_n)_e| \ = \ |d(e,[a_n,b_n])-(a_n\cdot b_n)_e|  \ \leq \ 6\delta,  \label{6delta}  
\end{equation}
with the inequality being by Lemma~\ref{6 delta}.  By the $(\delta)$-hyperbolicity condition,
\begin{align*}
(a_n\cdot b_n)_e \ & \geq \ \min\{(a_n\cdot a_n')_e,(a_n'\cdot b_n')_e,(b_n\cdot b'_n)_e\}-2\delta, \textup{  and} \\
(a'_n\cdot b'_n)_e \ & \geq \ \min\{(a_n\cdot a_n')_e,(a_n\cdot b_n)_e,(b_n\cdot b'_n)_e\}-2\delta. 
\end{align*}
 As $n \to \infty$ both  $(a_n\cdot a'_n)_e \to \infty$ and $(b_n\cdot b'_n)_e \to \infty$, but   $\limsup(a_n\cdot b_n)_e$ and $\limsup(a_n'\cdot b_n')_e$ are bounded above by $(a\cdot b)_e+1$ (else, passing to subsequences, we can assume $(a_n\cdot b_n)_e>(a\cdot b)_e+ 1/2$ for all $n$, and so $\liminf(a_n\cdot b_n)_e\geq(a\cdot b)_e+1/2$ contrary to the definition  of $(a\cdot b)_e$). So these two inequalities together  give $|(a_n\cdot b_n)_e-(a_n'\cdot b_n')_e|\leq 2\delta$ for all sufficiently large $n$.  Combining this with \eqref{6delta} gives the result.   
 \end{proof}

\emph{Visual metrics} are natural  metrics on the boundary $\partial X$ of a $(\delta)$-hyperbolic space $X$.  Their
essence is that $a,b \in \partial {X}$ are close when geodesics from a basepoint $e \in X$ to $a$ and from $e$ to $b$ fellow travel for a long distance.  One might hope that if $r >1$, then  $d(a,b) =  r^{-(a \cdot b)_e}$ would define such a metric, but unfortunately, as such, transitivity can fail.  Instead, say that a metric $d$ on  $\partial X$ is a \emph{visual metric} with \emph{visual parameter} $r > 1$ when there exist $k_1, k_2 >0$ such that for all $a,b \in \partial {X}$, 
\begin{equation}
k_1 r^{-(a \cdot b)_e} \ \leq \  d(a,b)  \ \leq  \  k_2 r^{-(a \cdot b)_e}. \label{visual eq}
\end{equation}  

\begin{lemma} \label{visual}
Suppose $X$ is a $(\delta)$-hyperbolic space, $r>1$, and  $e \in X$ is the base point with respect to which $\partial X$ is defined.   Then there is a visual metric on $\partial {X}$ with parameter $r$.   Moreover, any two visual metrics $d$ and $d'$ on $\partial {X}$  (perhaps with different $r$ and  $e$) are H\"older-equivalent in that  there exists $\alpha>0$ such that  the identity map $(\partial{X},d)\to(\partial{X},d')$ is $\alpha$-H\"older and its inverse is $(1/ \alpha)$-H\"older.    In particular, the \emph{visual topology} on $\partial X$ is independent of these choices. 
\end{lemma}

 The existence claim is \cite[Proposition~3.21, page 435]{BH1}.  H\"older-equivalence is known (see \cite[Theorem~2.18]{KB}), and follows, in this generality, immediately from the definition of visual metric and the fact that $(a\cdot b)_e\geq (a\cdot b)_{e'}-d(e,e')$.

We will need  that $\overline{X} := X\cup\partial X$ is a compactification of $X$:
\begin{lemma}\label{compactification}
If $X$ is a proper $(\delta)$-hyperbolic geodesic metric space, then there is a unique compact metrizable topology on $\overline{X}:=X\cup\partial X$ such that: the inclusions of $X$ and of $\partial X$ are homeomorphic onto their images, $\partial X$ is closed, and a sequence $x_n$ in $X$ converges to $x\in\partial X$ if and only if $(x_n)$ is in the equivalence class  $x$.
\end{lemma}
\begin{proof}
Uniqueness follows from the fact that for metrizable spaces, the topology is determined
by knowledge of which sequences converge and to which points they converge.  For a sequence $x_n$ in $\overline{X}$, $x_n\to x\in X$ in the topology of $\overline{X}$ precisely when all but finitely many $x_n\in X$ and $x_n\to x$ in the topology of $X$.  Also, $x_n\to x\in\partial X$ in the topology of $\overline{X}$ precisely when: (i) the subsequence of $x_n$ consisting of points in $\partial X$ is finite or converges to $x$ in the topology of $\partial X$, AND (ii) the subsequence of $x_n$ consisting of points in $X$ is finite or represents the equivalence class $x$.

The topology on $\overline{X}$ is constructed and sequential compactness is proved in  \cite[page~430,  III.H.3.7]{BH1} and  metrizability in  \cite[page~433 III.H.3.18(4)]{BH1}.   For metric spaces, sequential compactness is equivalent to compactness.
The agreement with the topology on $\partial X$ coming from the visual metric is \cite[page~435 III.H.3.21]{BH1}. 
The characterization of convergent sequences
follows from \cite[page~431, III.H.3.13]{BH1}.  
\end{proof}

Given a finitely generated group $G$ with finite generating set $A$, one forms the Cayley graph $C_A(G)$ with vertex set $G$ and
edge set $\{\{v,w\}\,|\,v,w\in G,\, vw^{-1}\in A\cup A^{-1}\}$.  The graph metric on $C_A(G)$ induces the \emph{word metric} on the vertex set $G$ and
$G$ is called a $(\delta)$-\emph{hyperbolic group} if the metric space $G$ is $(\delta)$-hyperbolic.  

\begin{lemma}\label{hyperbolic group compactification}
If $G$ is a hyperbolic group, then there is a unique compact metrizable topology on $\overline{G}:=G\cup\partial G$ such that: the inclusions of $G$ and of $\partial G$ are homeomorphic onto their images, $\partial G$ is closed, and a sequence $g_n$ in $G$ converges to $g\in\partial G$ if and only if $(g_n)$ is in the equivalence class $g$.
\end{lemma}
\begin{proof}
In this case, $C_A(G)$ is a proper hyperbolic geodesic metric space, to which Lemma~\ref{compactification} applies.
One may identify $\partial G$ with $\partial C_A(G)$ as visual metric spaces and thus identify $\overline{G}:=G\cup\partial G$ with a subspace of $\overline{C_A(G)}=C_A(G)\cup\partial G$, and so endow $\overline{G}$ with the subspace topology. 
The stated properties of $\overline{G}$ are now a consequence of Lemma~\ref{compactification}, and uniqueness follows as it does in the proof of Lemma~\ref{compactification}.
\end{proof}

Changing the finite generating set $A$ induces a quasi-isometry and so does not affect whether $G$ is hyperbolic  (Theorem~1.9 of  \cite[page~402,  III.H.1]{BH1}), nor does it affect the topology on $\partial{G}$ (Theorem~3.9 of  \cite[page~430,  III.H.3.7]{BH1}) or $\overline{G}$ (Lemma~\ref{lemma:Mitra}\ref{exists and cts}$\iff$\ref{extension} below). 

\section{Cannon--Thurston maps and Mitra's Lemma} \label{Mitra's lemma section}

Given two hyperbolic groups $\Lambda$ and $\Gamma$ and an injective homomorphism $\imath:\Lambda\to\Gamma$, one may ask if it extends to a continuous map $\ibar:\overline{\Lambda} \to \overline{\Gamma}$.  Equivalently (Lemma~\ref{lemma:Mitra}), one may ask whether $\imath$ induces a well-defined map $\ihat:\partial\Lambda\to\partial\Gamma$ sending $[(g_n)]$ to $[(\imath g_n)]$.  When one (and hence both) exist, $\ihat$ is the restriction of $\ibar$ to $\partial\Lambda$ and is called the \emph{Cannon--Thurston map}.  The next section will prove the existence of Cannon--Thurston maps for the heavily distorted free subgroups of hyperbolic hydras.  This section deals with general tools for showing the existence of Cannon--Thurston maps. 

For now, let us deal with a more general setting.  Consider a map $f:Y\to X$ where $(X, d_X)$ is $(\delta_X)$--hyperbolic and $(Y, d_Y)$ is $(\delta_Y)$--hyperbolic.  We will assume $X$ and $Y$ are proper and geodesic, or else Lemma~\ref{compactification} will not provide us a topology on $\overline{X}$ and $\overline{Y}$ (though one could still ask whether $(x_n)\to(f(x_n))$ gives a well-defined continuous map $\partial X\to\partial Y$). The following lemma is an embellished version of Mitra's criterion for the Cannon--Thurston map  to exist (\cite[Lemma 2.1]{CTnormal} and \cite[Lemma 2.1]{CTtrees}).   

Our notation is that $B_X(e,R)  =  \set{x \in X \mid d_X(e,x) <r }$ and $\wb{B}_X(e,R) =  \set{x \in X \mid d_X(e,x) \leq r }$, and we write $\gamma = [x,y]_X$ to mean \emph{$\gamma$ is a geodesic in $X$ from $x$ to $y$}.   The metrics on $\partial X$ and $\partial Y$ implicit in this lemma are any visual metrics $d_{\partial X}$ and $d_{\partial Y}$.

The additional hypothesis for \ref{Baker M} can be removed if $\delta_Y$, the constant of hyperbolicity for $Y$, is zero; for $\delta_Y>0$, we do not know whether \ref{Baker M} is equivalent to or strictly weaker than \ref{exists}--\ref{Mitra M} in its absence.   An inclusion map of a subgroup into an ambient group is Lipschitz when both have word metrics coming from some finite generating sets, and so it is satisfied in that setting.

We remark that the equivalence \ref{exists}--\ref{product M} also follows for $X$ and $Y$ hyperbolic groups and $f:X\to Y$ an injective homomorphism, with the same proof (replacing Lemma~\ref{compactification} with Lemma~\ref{hyperbolic group compactification}).

\begin{lemma}
\label{lemma:Mitra}
Suppose $(X, d_X)$ and $(Y, d_Y)$ are infinite proper geodesic hyperbolic metric spaces, and $f:Y \to X$ is a proper map.    
Fix a basepoint $e \in Y$.  Define $M, M', M'' : [0,\infty) \to [0,\infty)$ by 
\begin{align*}
M(N) & : = \ \inf \set{ (f(x) \cdot f(y))^X_{f(e)}  \mid  x, y \in Y \textup{ and } (x \cdot y)^Y_e \geq N }, \\
M'(N) & : = \  \inf \set{  d_X(f(e), \gamma)   \mid \gamma = [f(x),f(y)]_X \textup{ for some }   [x, y]_Y \textup{ in } Y \ssm B_Y(e,N)    }, \\
M''(N) & : = \ \inf \set{  d_X(f(e), \gamma)   \mid \gamma  = [f(z),f(y)]_X \textup{ for some } z \textup{ on some } [e,y]_Y \textup{ with } d_Y(e,z) \geq N   }.  
\end{align*}  
The following are equivalent:
\renewcommand{\theenumi}{\textup{(\alph{enumi})}}
\vspace*{-6pt}
\begin{enumerate}
 \item   $(a_n)\mapsto(f(b_n))$ induces a well-defined function $\fhat:\partial{Y}\to\partial{X}$. \label{exists}
  \item  $(a_n)\mapsto(f(b_n))$ induces a well-defined,  continuous map $\fhat:\partial{Y}\to\partial{X}$.   \label{exists and cts}  
	\item  There exists a continuous extension $\overline{f}:\overline{Y}\to\overline{X}$ of $f$.  \label{extension}
 \item  $M(N)\to\infty$ as $N\to\infty$. \label{product M} 
 \item  $M'(N)\to\infty$ as $N\to\infty$. \label{Mitra M}
\end{enumerate}
\vspace*{-6pt}
Moreover, if  $\sup\{ d_X(f(x), f(y)) \mid d_Y(x,y) \leq r \} < \infty$ for all $r \geq 0$, then these are also equivalent to
\renewcommand{\theenumi}{\textup{(\alph{enumi})}}
\vspace*{-6pt}
\begin{enumerate}
\addtocounter{enumi}{5}
  \item  $M''(N)\to\infty$ as $N\to\infty$. \label{Baker M}
\end{enumerate} 
 \end{lemma}

\begin{proof}
Each of $M(N)$, $M'(N)$, and $M''(N)$ is a non-decreasing function, so is either bounded or tends to $\infty$ as $N \to \infty$.

That \ref{exists and cts}$\implies$\ref{exists} is immediate.  

Here is why \ref{exists}$\implies$\ref{product M}. Suppose  $M(N)\leq C$ for all $N$.
So there are sequences $(p_n)$ and $(q_n)$   in $Y$ with $(p_n\cdot q_n)^Y_{e}\to\infty$ but $( f(p_n) \cdot f(q_n))^X_{f(e)} \leq C$ for all $n$.   As $Y \cup \partial Y$ is sequentially compact by Lemma~\ref{compactification}, 
both $(p_n)$ and $(q_n)$ have  subsequences which converge in $Y\cup\partial Y$.   But the condition $(p_n\cdot q_n)^Y_{e}\to\infty$ precludes any such subsequence from converging in $Y$, so those subsequences converge  to points  in $\partial Y$, indeed to the same point.

Next we  prove \ref{product M}$\implies$\ref{exists and cts}.  Suppose sequences $(p_n)$ and $(q_n)$  in $Y$ both represent  the same point in $\partial Y$. Then $(p_n \cdot q_n)_e^Y\to\infty$ as $n\to\infty$, and so $(f(p_n) \cdot f(q_n))_{f(e)}^X\to\infty$, since $M(N)\to\infty$ as $N\to\infty$.   Thus if $(f(p_n))$ and $(f(q_n))$ represent  points in $\partial X$, then those points are the same.

So, to prove $\fhat$ is well-defined, it suffices to show that if a sequence $(a_n)$ in $Y$ represents a point in $\partial Y$ (and so $d_Y(e,a_n)     \to\infty$, since $d_Y(e,a_n)\geq(a_n\cdot a_m)_e^Y$), then $(f(a_n))$ represents a point in $\partial X$.  Indeed, it suffices to show that a \emph{subsequence} of $(f(a_n))$ represents a point in $\partial X$.   
By sequential compactness of $X\cup\partial X$  (Lemma~\ref{compactification}), 
a  subsequence of $(f(a_n))$ converges. If it converges to a point in $X$, then a subsequence of $(f(a_n))$ is in some compact (by properness of $X$) ball $\wb{B}_X(e,R)$. 
But then, by properness of $f$, a subsequence of  $(a_n)$ would be contained in some ball $B_Y(e,R')$, which would contradict  $d_Y(e,a_n)\to\infty$.   So some subsequence of $(f(a_n))$ converges to (that is, \emph{represents}---see Lemma~\ref{compactification}) 
a point in $\partial X$.

To establish continuity, suppose $p,q\in\partial Y$.   
By definition of the visual metrics  $d_{\partial X}$ and $d_{\partial Y}$,  
there exist constants $r, s >1$ and $k,l >0$ (independent of $p,q$)  such that   
\begin{align} 
d_{\partial X}(\fhat(p) , \fhat(q)) &  \ \leq \ k r^{- ( \fhat(p) \cdot \fhat(q))_{f(e)}^X}      \label{boundX} 
 \end{align}
 and
\begin{align} 
d_{\partial Y}(p,q) & \ \geq \ l s^{- ( p \cdot q)_{e}^Y}.      \label{boundY} 
 \end{align}
 
Since $(p\cdot q)_e^Y=\sup \liminf_{ m,n \to \infty}{(p_m \cdot q_n)^Y_e}$, there exist sequences $(p_m)$ and $(q_n)$ in $Y$ representing $p$ and $q$, respectively, with $\liminf_{m,n\to\infty}{(p_m\cdot q_n)_e^Y}\geq (p\cdot q)_e^Y-1$.
So  $(p_m\cdot q_n)_e^Y\geq (p,q)_e^Y-2$ for all sufficiently large $m,n$.
By definition of $M$ we have $(f(p_m)\cdot f(q_n))_{f(e)}^X\geq M((p \cdot q)_e^Y-2)$  for such $m,n$, and hence 
\begin{align} 
(\fhat(p)\cdot\fhat(q))_{f(e)}^X & \ \geq \  M((p \cdot q)_e^Y-2). \label{X by Y}
\end{align}

Combining \eqref{boundX} and \eqref{X by Y}, we have  
\begin{align}
d_{\partial X}(\fhat(p) , \fhat(q)) & \ \le  \  k r^{- M \left( ( p \cdot q)_e^{Y} -2 \right)}.  \label{boundX2} 
\end{align}  
So, by \eqref{boundY}, if we make $d_{\partial Y}(p,q)$ sufficiently small, we can make $(p \cdot q)_e^Y$ arbitrarily large,   so by hypothesis make  $M \left( (p \cdot q)_e^Y -2\right)$ arbitrarily large,  and so by \eqref{boundX2} make 
$d_{\partial X}(\fhat(p) , \fhat(q))$ arbitrarily small.    Thus $\fhat$ is continuous. 

That \ref{extension}$\implies$\ref{exists and cts} is an immediate consequence of Lemma~\ref{compactification},  
for we just take $\fhat$ to be the restriction of $\overline{f}$.  Properness of $f$ guarantees that $\fbar(\partial Y)\subseteq\partial X$.

To see that \ref{exists and cts}$\implies$\ref{extension} we show that the function $\fbar:=f\cup\fhat$ is continuous.
Since $\overline{X}$ and $\overline{Y}$ are metric spaces, it suffices to show that $p_n\to p$ implies $\fbar(p_n)\to\fbar(p)$ whenever $p_n,p\in\overline{Y}$.  Since $Y$ is an open subset of $\overline{Y}$ on which $\fbar$ restricts to the continuous function $f$, we may assume $p\in\partial Y$.  Since $\fhat$ is continuous, we may assume each $p_n\in Y$.  But then Lemma~\ref{compactification} 
says $(p_n)$ \emph{represents} $p$, so by \ref{exists and cts}, $(f(p_n))$ represents $\fhat(p)$.  Using Lemma~\ref{compactification} 
again, we see $f(p_n)\to\fhat(p)$.  That is, $\fbar(p_n)\to\fbar(p)$.

The equivalence of \ref{product M} and \ref{Mitra M} comes from Lemma~\ref{6 delta}, which implies that  
there exists $C>0$ such that $M'(N) \leq M(N + C) + C$ and $M(N) \leq M'(N + C) + C$ for all $N$.   
 
That \ref{Mitra M}$\implies$\ref{Baker M} is immediate as  $[z,y]_Y$ is a  geodesic segment in $Y$ lying outside $B_Y(e, N)$.\

Here is a proof that \ref{Baker M}$\implies$\ref{Mitra M} under the assumption that  $\sup\{ d_X(f(x), f(y)) \mid d_Y(x,y) \leq r \} < \infty$ for all $r \geq 0$.  
  Suppose $\lambda=[h_1,h_2]_Y$.  As $t:=(h_1 \cdot h_2)^Y_e$   approximates $d_Y( \lambda ,e)$ with error at most
  a constant (Lemma~\ref{6 delta}), it is enough to show $d_X([ f(h_1),f(h_2)]_X,e)\to\infty$ as $t\to\infty$.  Let $\alpha_i = [e,h_i]_Y$  for $i=1,2$.      By the slim-triangles condition, $[f(h_1),f(h_2)]_X$ lies in a $C$-neighborhood of a piecewise-geodesic path $[f(h_1),f(\alpha_1(t))]_X\cup[f(\alpha_1(t)),f(\alpha_2(t))]_X\cup[f(\alpha_2(t)),f(h_2)]_X$ for some constant $C$.   So it is enough to show that the distance  of each of these three segments from $f(e)$  in $X$ tends to $\infty$ as $t \to \infty$.  
 This is so for  $[f(h_1),f(\alpha_1(t))]_X$ and $[f(\alpha_2(t)),f(h_2)]_X$ by \ref{Baker M}.  
By the thin-triangles condition  (see e.g.\   \cite[pages 408--409]{BH1}), $d_Y(\alpha_1(l), \alpha_2(l))$ is at most a constant for all $0 \leq l \leq t$, and so in particular   $d_Y(\alpha_1(t),\alpha_2(t))$ is at most a constant.  
The  assumption that  $\sup\{ d_X(f(x), f(y)) \mid d_Y(x,y) \leq r \} < \infty$ for all $r \geq 0$ gives an upper bound, independent of $t$, on the length of $[f(\alpha_1(t)),f(\alpha_2(t))]_X$.
Since the distances of the endpoints of this segment from $f(e)$ in $X$ tend to $\infty$ as $t \to \infty$, so does the distance of  the whole segment.
\end{proof}

The first part of the following lemma shows that Theorem~\ref{existence thm} is not a quirk of the choice of generating sets.  The second establishes the sense in which the function $\varepsilon(\delta)$ of Section~\ref{intro} is an invariant for Cannon--Thurston maps.  The third will allow us to reinterpret  Lemma~\ref{lemma:Mitra}  (as Corollary~\ref{cor:2point3}) in a manner well suited to analyzing  hyperbolic hydra.  

\begin{lemma} \label{equiv existence}
Suppose $\Lambda$ is a hyperbolic subgroup of a hyperbolic group $\Gamma$.  
\vspace*{-6pt}
\begin{enumerate}
\renewcommand{\theenumi}{\textup{(\roman{enumi})}}
\item \label{gen set} Whether the  Cannon--Thurston map $\partial \Lambda \to \partial \Gamma$ exists does not depend on the choice of finite generating sets giving the word metrics. 

\item \label{delta independence} 
If the Cannon--Thurston map $\imath : \partial \Lambda \to \partial \Gamma$ exists  for a hyperbolic subgroup $\Lambda$ of a hyperbolic group $\Gamma$,   the modulus of continuity  for $\imath$  does not depend on the finite generating sets and the choices of visual metrics  up to the following H\"older-type  equivalence.  If $\varepsilon(\delta)$  and  $\varepsilon'(\delta)$ are the moduli of continuity of  $\imath$ defined with respect to different such choices, then there are functions $f_1, f_2 : (0,\infty) \to (0,\infty)$, each of the form $x \mapsto C_ix^{\alpha_i}$ for some $C_i,\alpha_i>0$, such that $\varepsilon'(\delta) \leq (f_1 \circ \varepsilon \circ f_2) (\delta)$ for all $\delta >0$.

\item \label{Cayley graphs} Suppose $A$ and $B$ are  finite generating sets  for $\Gamma$ and  $\Lambda$, respectively.   Suppose   $f :  C_B(\Lambda)  \to   C_A(\Gamma)$ is any map between the respective Cayley graphs which restricts to the inclusion $\Lambda \hookrightarrow \Gamma$ on the vertices of $C_B(\Lambda)$ and sends  edges to geodesics in $C_A(\Gamma)$.   The  Cannon--Thurston map $\partial \Lambda \to \partial \Gamma$, defined in terms of finite generating sets $A$ for $\Gamma$ and $B$ for $\Lambda$, exists  if and only if   the Cannon--Thurston map  $\partial  C_B(\Lambda)  \to \partial  C_A(\Gamma)$ does. 
\end{enumerate}  
 \end{lemma}
 
\begin{proof}
Suppose $\partial_1 \Gamma$, $\partial_2 \Gamma$, $\partial_1 \Lambda$, and $\partial_2 \Lambda$ are boundaries of $\Gamma$ and $\Lambda$ defined with respect to different finite generating sets  and the Cannon--Thurston map $\partial_1\Lambda\to\partial_1\Gamma$ exists.
 The identity maps on $\Lambda$ and $\Gamma$ changing the word metrics are Lipschitz, so induce maps $\partial_2\Lambda\to\partial_1\Lambda$ and $\partial_1\Gamma\to\partial_2\Gamma$ (Theorem~3.9 of  \cite[page~430,  III.H.3.7]{BH1}).  The composite map $\partial_2\Lambda\to\partial_1\Lambda\to\partial_1\Gamma\to\partial_2\Gamma$ satisfies Lemma~\ref{lemma:Mitra}\ref{exists}, so the Cannon--Thurston map $\partial_2\Lambda\to\partial_2\Gamma$ exists.

If $X,Y,Z$ are metric spaces and $g:Y\to Z$ and $h:X\to Y$ are maps with moduli of continuity $\varepsilon_g$ and $\varepsilon_h$, respectively, then it follows  from the definition that $\varepsilon_{g\circ h}(\delta)\leq (\varepsilon_{g}\circ\varepsilon_{h})(\delta)$ for all $\delta >0$.  Specializing to the case of $\partial_2 \Lambda \to \partial_1 \Lambda \to \partial_1 \Gamma \to \partial_2 \Gamma$ in the previous paragraph, \ref{delta independence} follows from the fact (Proposition~5.5 and Theorem~6.5 of \cite{BS1}) that $\partial_2 \Lambda \to \partial_1 \Lambda$ and $\partial_1 \Gamma\to\partial_2 \Gamma$ are H\"older.   

For \ref{Cayley graphs},  a Cannon--Thurston map $\overline{C_B(\Lambda)}\to\overline{C_A(\Gamma)}$ restricts to a Cannon--Thurston map $\overline{\Lambda}\to\overline{\Gamma}$.  Conversely, the quasi-isometries $\Gamma\hookrightarrow C_A(\Gamma)$ and $\Lambda\hookrightarrow C_B(\Lambda)$ induce isometries $\partial\Gamma\to\partial C_A(\Gamma)$ and $\partial C_B(\Lambda)\to\partial\Lambda$.  So a Cannon--Thurston map $\partial\Lambda\to\partial\Gamma$ induces a composite map $\partial C_B(\Lambda)\to\partial\Lambda\to\partial\Gamma\to\partial C_A(\Gamma)$ satisfying Lemma~\ref{lemma:Mitra}\ref{exists}.
\end{proof}

\begin{cor}
\label{cor:2point3}
Suppose $\Lambda$ is a  finite-rank  free subgroup of a hyperbolic group $\Gamma$.  Suppose $A$ is a finite generating set  for $\Gamma$ and $B$ is a free basis for  $\Lambda$.   The Cannon--Thurston map $\partial \Lambda \to \partial \Gamma$ exists if and only if   for all $M''>0$, there exists $N$ such that  whenever  $\alpha \beta$ is a reduced word on $B$ with $|\alpha| \geq N$,  every geodesic in the Cayley graph $C_A(\Gamma)$ joining $\alpha$ to $\alpha \beta$ lies outside the ball of radius $M''$ about $e$.
 \end{cor}

\begin{proof}
By Lemma~\ref{equiv existence}\ref{Cayley graphs}, the Cannon--Thurston map $\partial \Lambda \to \partial \Gamma$ exists if and only if  
  the Cannon--Thurston map  $\partial  C_B(\Lambda)  \to \partial  C_A(\Gamma)$ for $f$ (as defined in that lemma) does.  Now 
 applying condition \ref{Baker M} of  Lemma~\ref{lemma:Mitra} to $f$ gives the result, since reduced words correspond to geodesics in $C_B(\Lambda)  $.
\end{proof}

\section{Cannon--Thurston maps for hyperbolic hydra groups} \label{CT for hydra}

The \emph{hyperbolic hydra} $\Gamma_k$  are a family of hyperbolic groups with  distorted (when $k>1$) free subgroups $\Lambda_k$.  In this section we will review some pertinent details from \cite{HypHydra} and  \cite{Hydra}   of the construction and properties of the hyperbolic hydra groups $\Gamma_k$.  We then show the existence of Cannon--Thurston maps $\partial\Lambda_k\to\partial\Gamma_k$ for hyperbolic hydra.  Throughout, we fix an integer $k \geq 1$.

The hyperbolic hydra $\Gamma_k$ of \cite{HypHydra} is an elaboration of the \emph{hydra group} $G_k$ of \cite{Hydra}: 
\begin{align*} 
G_k & \  = \  F(a_1, \ldots, a_k) \rtimes_{\phi}  \Z   
\end{align*}
where  $\phi$ is the automorphism     
$$\phi(a_i) \ = \    \begin{cases} 
         \ a_1    &     i=1, \\
         \ a_i a_{i-1}  &   1 < i \leq  k
         \end{cases}   $$
 of the free group $F_k = F(a_1, \ldots, a_k)$.
     Let    $t$ denote a generator of the $\Z$-factor, so $\phi(a_i) = t^{-1} a_i t$ in $G_k$ for all $i$.   
In \cite{Hydra} it is proved that $G_k$ is CAT(0) and has a  rank-$k$ free subgroup $H_k  =  \langle a_1t, \ldots, a_kt \rangle$, distorted  so that $\Dist^{G_k}_{H_k} \simeq A_k$.

Since the restriction of $\phi$ to $\langle a_1,\ldots,a_{i-1}\rangle = F_{i-1}$ is an automorphism for each $i$ and $\phi(a_i)\in a_i\langle a_1,\ldots,a_{i-1}\rangle$, we have
\begin{lemma} \label{what phi does}
For any integer $j$ (positive or negative), 
$$\phi^{j}(a_i)\in{a_i}\langle a_1,\ldots,a_{i-1}\rangle\textrm{\qquad and \qquad}\phi^{j}(a_i^{-1})\in\langle a_1,\ldots,a_{i-1}\rangle a_i^{-1}.$$
\end{lemma}
For example, $\phi^{-1}(a_7)=a_7a_5a_3a_1a_2^{-1}a_4^{-1}a_6^{-1}$.

The \emph{normal form} of $g$ in $G_k$ is the unique $\wh{w} t^m$ 
such that $\wh{w}$ is a reduced word on $a_1, \ldots, a_k$ and $g = \wh{w} t^m$ in $G_k$.
For any $r\leq k$, an \emph{$H_r$-word} is a reduced word on $ a_1t, \ldots, a_rt$.  For example,    $(a_3t)(a_3t)(a_2t)(a_3t)^{-1}$ is an $H_3$-word and its normal form is  $a_3a_3(a_1^{-1}a_2a_1^2a_2a_1^{-1})a_3^{-1}t^2$, since in $G_3$ 
\begin{eqnarray*}
(a_3t)(a_3t)(a_2t)(a_3t)^{-1}&=&a_3(ta_3t^{-1})(t^2a_2t^{-2})(t^2{a_3^{-1}}t^{-2})t^2 \\
 &=&a_3\phi^{-1}(a_3)\phi^{-2}(a_2)\phi^{-2}(a_3^{-1})t^2\\
 &=&a_3(a_3a_1a_2^{-1})(a_2a_1^{-2})(a_2a_1^2a_2a_1^{-1}a_3^{-1})t^2\\
 &=&a_3a_3(a_1^{-1}a_2a_1^2a_2a_1^{-1})a_3^{-1}t^2.
\end{eqnarray*}

Lemma 6.1 of \cite{Hydra}, which says that $H_k\cap\langle t\rangle=\set{1}$ in $G_k$, implies the following two lemmas:

\begin{lemma} \label{uniqueness of exponent}
Given $g\in G_k$,  if there exists $j$ such that   $g  t^{j} \in H_k$, then that $j$ is unique. 
\end{lemma}

\begin{lemma} \label{uniqueness from normal form}  
 If $g_1, g_2 \in H_k$   have normal forms $\wh{w} t^{n_1}$ and $\wh{w} t^{n_2}$, respectively, then $g_1 =g_2$.  
\end{lemma}

Observe that the relative locations of the   $(a_3t)^{\pm 1}$ in $(a_3t)(a_3t)(a_2t)(a_3t)^{-1}$ are the same as  the relative locations of the $a_3^{\pm 1}$  in $a_3a_3(a_1^{-1}a_2a_1^2a_2a_1^{-1})a_3^{-1}$.  
That is, ignoring all other symbols, the former word has two $(a_3t)$ symbols followed by $(a_3t)^{-1}$, while the latter has two $a_3$ symbols followed by $a_3^{-1}$.  This is an instance of the following lemma.

\begin{lemma}\label{normal form of h word}
Consider an  $H_r$-word  $w$ of the form  $$u_0(a_rt)^{\epsilon_1}u_1(a_rt)^{\epsilon_2}u_2\cdots(a_rt)^{\epsilon_n}u_n$$
where $u_0,\ldots,u_n\in H_{r-1}$ and $\epsilon_1,\ldots,\epsilon_n\in\{\pm1\}$, and  $u_i\neq1$ whenever $\epsilon_i=-\epsilon_{i+1}$.  Then, for all $s \in \Z$, the normal form $\wh{w} t^m$ of $w$ satisfies
$$\phi^s(\wh{w})=v_0a_r^{\epsilon_1}v_1a_r^{\epsilon_2}v_2\cdots a_r^{\epsilon_n}v_n$$
for some $v_0,\ldots,v_n\in\langle a_1,\ldots,a_{r-1}\rangle$ where $v_i\neq1$ whenever $\epsilon_i=-\epsilon_{i+1}$.
\end{lemma}

The case $s=0$ is Lemma 6.2 of \cite{Hydra}.  We will not need to use Lemma \ref{normal form of h word}; we state it   because it sets the scene for the  analogous Lemma~\ref{normal form of lambda word} below.  A proof of \ref{normal form of h word} can be extracted from the proof  we will give for \ref{normal form of lambda word} by replacing $\Gamma_k,\Lambda_r,\Lambda_{r-1},\phi$ with $G_k,H_r,H_{r-1},\theta$, respectively, and invoking Lemmas \ref{what phi does} and \ref{uniqueness from normal form} instead of   \ref{what theta does} and \ref{uniqueness from normal form hyperbolic}, respectively. 

The construction  of $G_k$  above is elaborated in \cite{HypHydra} to give the hyperbolic hydra $\Gamma_k$.     
It involves additional variables $a_0,b_1,\ldots,b_l$ and has the   form
$$\Gamma_k  \ = \ F \rtimes_{\theta} \Z$$ where  $F$ is the free group $F(a_0, \ldots, a_k, b_1, \ldots, b_l)$,   and 
  $\theta$ is an  automorphism of $F$ whose restriction  to $F(b_1, \ldots, b_l)$ is an automorphism and
   $$\theta(a_i)  \ = \  \begin{cases} 
         \ U a_1 V   &  i=0, \\  
         \ a_0    &     i=1, \\
         \ a_i a_{i-1}  &   1 < i \leq  k, 
         \end{cases}   $$
where  $U$ and $V$ are words  on $b_1, \ldots, b_l$.  We will prove here that  Cannon--Thurston maps exist for all hyperbolic  $\Gamma_k$ of this form.  In  \cite{HypHydra}, $U$, $V$, $l$ and   $\theta |_{F(b_1, \ldots, b_l)}$ are chosen carefully to ensure $\Gamma_k$ is hyperbolic.  (In fact, in \cite{HypHydra},  $l=17$, and  $U$ and $V$ depend on $k$, but  $\theta |_{F(b_1, \ldots, b_l)}$ does not.) 
         
Let $t$ denote a generator of the $\Z$-factor in $\Gamma_k  \ = \ F \rtimes_{\theta} \Z$, so $t^{-1} a_i t = \theta(a_i)$ and $t^{-1} b_j t = \theta(b_j)$ for all $i$ and $j$.
For $1 \leq r \leq k$, let $\Lambda_r $ be the subgroup $ \langle a_0t,   \ldots, a_r t, b_1, \ldots, b_l \rangle$ of $\Gamma_k$.  It is proved in \cite{HypHydra} that $\Lambda_k$ is free of rank $k+l+1$ and is distorted so that $\Dist^{\Gamma_k}_{\Lambda_k} \succeq A_k$.    Let $\Lambda_0=\langle b_1,\ldots, b_l \rangle$.  (Note that $\Lambda_r$ actually depends on both $k$ and $r$, but $k$ is fixed throughout this section.)   

Understand  $\theta^n(\wh{w})$ to mean the reduced  word  on $a_0, \ldots, a_k, b_1, \ldots, b_l$ that represents $\theta^n(\wh{w})$ in $F$. 

Mapping   $a_i \mapsto a_{\max \{1,i\}}$, $b_j \mapsto 1$ and $t \mapsto t$ for all $i,j$ defines a surjection $\Gamma_k \onto G_k$ such that $\Phi(\Lambda_r)=H_r$.  

Corresponding to Lemmas~\ref{what phi does}--\ref{normal form of h word} for $G_k$, we have the following Lemmas~\ref{what theta does}--\ref{normal form of lambda word} for $\Gamma_k$.
 
The definition of $\theta$ immediately gives:
\begin{lemma} \label{what theta does}
For any integer $j\in\mathbb{Z}$ (positive or negative), and any $i>1$
$$\theta^{j}(a_i)\in{a_i}\langle a_0,a_1,\ldots,a_{i-1},b_1,\ldots,b_l\rangle\textrm{\qquad and \qquad}\theta^{j}(a_i^{-1})\in\langle a_0,a_1,\ldots,a_{i-1},b_1,\ldots,b_l\rangle a_i^{-1}.$$
\end{lemma}

The \emph{normal form} of $g \in \Gamma_k$ is the unique $\wh{g} t^m$ 
such that $\wh{g}$ is a reduced word on $a_0, a_1, \ldots, a_k,b_1,\ldots,b_l$ and $g = \wh{g} t^m$ in $\Gamma_k$.
For any $1\leq r\leq k$, a \emph{$\Lambda_r$-word} is a reduced word on $a_0t, a_1t, \ldots, a_rt,b_1,\ldots,b_l$.  Likewise, a $\Lambda_0$-word is a reduced word on $b_1,\ldots,b_l$.

Proposition 4.8 of \cite{HypHydra} says that $\Lambda_k\cap\langle t\rangle=\set{1}$. So we immediately have the following analogues of Lemmas~\ref{uniqueness of exponent} and \ref{uniqueness from normal form}: 

 \begin{lemma} \label{uniqueness of exponent hyperbolic}
Given $g\in \Gamma_k$,  if there exists $j$ such that   $g  t^{j} \in \Lambda_k$, then that $j$ is unique. 
\end{lemma}
 
 \begin{lemma} \label{uniqueness from normal form hyperbolic}
 If $g_1, g_2 \in \Lambda_k$   have normal forms $\wh{w} t^{n_1}$ and $\wh{w} t^{n_2}$, respectively, then $g_1 =g_2$.  
 \end{lemma}

Finally, we have analogues of Lemma~\ref{normal form of h word}.  We treat the $r>1$ and $r=1$ cases separately.   First--

\begin{lemma}\label{normal form of lambda word} 
Let $r>1$.  Consider a $\Lambda_r$--word $w$ of the form $$u_0(a_rt)^{\epsilon_1}u_1(a_rt)^{\epsilon_2}u_2\cdots(a_rt)^{\epsilon_n}u_n$$
where $u_0,\ldots,u_n\in \Lambda_{r-1}$ and $\epsilon_1,\ldots,\epsilon_n\in\{\pm1\}$, and  $u_i\neq1$ whenever $\epsilon_i=-\epsilon_{i+1}$.
Then for any $s\in\mathbb{Z}$,
 $$\theta^s(\wh{w})=v_0a_r^{\epsilon_1}v_1a_r^{\epsilon_2}v_2\cdots a_r^{\epsilon_n}v_n$$
for some $v_0,\ldots,v_n\in\langle a_0,a_1,\ldots,a_{r-1},b_1,\ldots,b_l\rangle$  where $v_i\neq1$ whenever $\epsilon_i=-\epsilon_{i+1}$.
\end{lemma}

In short, Lemma~\ref{normal form of lambda word} says that for a reduced $\Lambda_r$ word $w$ and its normal form $\wh{w}t^n$,
the occurrences of $a_rt$ in $w$ correspond to occurrences of $a_r$ in $\wh{w}$ (and indeed in   $\theta^s(\wh{w})$ for any $s$) and the occurrences of $(a_rt)^{-1}$ in $w$ correspond to occurrences of $a_r^{-1}$ in $\wh{w}$ (and indeed in   $\theta^s(\wh{w})$ for any $s$): the count of each and the order in which they occur in their respective words is preserved.

\begin{proof}[Proof of Lemma~\ref{normal form of lambda word}]

Write each $u_j$ in normal form: $u_j=\wh{u_j}t^{m_j}$ with $\wh{u_j}\in\langle a_0,a_1,\ldots,a_{r-1},b_1,\ldots,b_l\rangle$. Then
\begin{equation}\label{whg hyp}
\theta^s(\wh{w})=\theta^s(\wh{u_0})\theta^{p_1}(a_r^{\epsilon_1})\theta^{q_1}(\wh{u_1})\theta^{p_2}(a_r^{\epsilon_2})\theta^{q_2}(\wh{u_2})\cdots\theta^{p_{n}}(a_r^{\epsilon_n})\theta^{q_n}(\wh{u_n})
\end{equation}
where $q_i=s-(m_0+m_1+\cdots+ m_{i-1})-(\epsilon_1+\cdots+\epsilon_{i})$ and where
$$p_i=\begin{cases}
s-(m_0+m_1+\cdots+ m_{i-1})-(\epsilon_1+\cdots+\epsilon_{i-1}) & \textrm{ if $\epsilon_{i}=1$}\\
s-(m_0+m_1+\cdots+ m_{i-1})-(\epsilon_1+\cdots+\epsilon_{i-1})+1 & \textrm{ if $\epsilon_{i}=-1$.}
\end{cases}$$

By Lemma \ref{what theta does}, the result follows unless $a_r^{\pm1}a_r^{\mp1}$ is cancelled during the reduction of the right side of equation (\ref{whg hyp}).
For an $a_r^{-1}a_r$ cancellation to occur in (\ref{whg hyp}), it would have to occur within $\theta^{p_i}(a_r^{-1})\theta^{q_i}(\wh{u_i})\theta^{p_{i+1}}(a_r)$ for some $i$.  But then Lemma~\ref{what theta does} would yield $\theta^{q_i}(\wh{u_i})=1$, so $\wh{u_i}=1$.  By Lemma~\ref{uniqueness from normal form hyperbolic}, we would have $u_i=1$, a contradiction. For an $a_ra_r^{-1}$ cancellation to occur in (\ref{whg hyp}), we would have $\theta^{p_i}(a_r)\theta^{q_i}(\wh{u_i})\theta^{p_{i+1}}(a_r^{-1})=1$ for some $i$ such that $\epsilon_i=1$ and $\epsilon_{i+1}=-1$.  But then $q_i-p_i=-1$ and $p_{i+1} -p_i=-m_i$.  So 
\begin{eqnarray*}
1 & = & \theta^{-p_i}(\theta^{p_i}(a_r)\theta^{q_i}(\wh{u_i})\theta^{p_{i+1}}(a_r^{-1})) \\
& = & a_r \theta^{q_i-p_i}( \wh{u_i}) \theta^{p_{i+1}-p_i}(a_{r}^{-1}) \\
& = & a_r \theta^{-1}( \wh{u_i}) \theta^{-m_i}(a_{r}^{-1}) \\
& = & a_r t  \wh{u_i} t^{-1} t^{m_i} a_r^{-1} t^{-m_i} \\
& = & a_r t u_i (a_r t)^{-1}  t^{-m_i}.
\end{eqnarray*}
So  $ t^{m_i} = (a_r t) u_i (a_r t)^{-1} \in \Lambda_k$.
But then $t^{m_i}=1$ by Lemma \ref{uniqueness from normal form hyperbolic}.
This would contradict the fact that $u_i\in \Lambda_{r-1}\smallsetminus\{1\}$.   
Thus no $a_r^{\pm1}a_r^{\mp1}$ cancellation occurs in the reduction of the right side of equation (\ref{whg hyp}) and the lemma is proved.
\end{proof}

Next we give the $r=1$ analogue to Lemma \ref{normal form of lambda word}.  Recall that $\theta(a_0)=Ua_1V$ and $\theta(a_1)=a_0$, so Lemma \ref{what theta does} does not apply in the $r=1$ case.  Roughly speaking, Lemma \ref{normal form of lambda word} could be expanded to accommodate the $r=1$ case by allowing occurrences of $a_0$ to sometimes swap with occurrences of $a_1$.  More precisely: 

\begin{lemma} \label{normal form of lambda word special case}
Consider a $\Lambda_1$--word $w$ of the form $$w=u_0(a_{\mu_1}t)^{\epsilon_1}u_1(a_{\mu_2}t)^{\epsilon_2}u_2\cdots(a_{\mu_n}t)^{\epsilon_n}u_n$$
where $u_0,\ldots,u_n\in \Lambda_{0}=\langle b_1,\ldots,b_l\rangle$ with $\mu_1,\ldots,\mu_n\in\{0,1\}$ and $\epsilon_1,\ldots,\epsilon_n\in\{\pm1\}$, and with $u_i\neq1$ whenever $\mu_i=\mu_{i+1}$ and $\epsilon_i=-\epsilon_{i+1}$.
Then for all $s\in\mathbb{Z}$,
 $$\theta^s(\wh{w})=v_0a_{\xi_1}^{\epsilon_1}v_1a_{\xi_2}^{\epsilon_2}v_2\cdots a_{\xi_n}^{\epsilon_n}v_n$$
for some $v_0,\ldots,v_n\in\langle b_1,\ldots,b_l\rangle$ and some $\xi_1,\ldots,\xi_n\in\{0,1\}$
with $v_i\neq1$ whenever $\xi_i=\xi_{i+1}$ and $\epsilon_i=-\epsilon_{i+1}$.
\end{lemma}

\begin{proof}
We have
\begin{equation}\label{whg hyp special case}
\theta^s(\wh{w})=\theta^s(u_0)\theta^{p_1}(a_{\mu_1}^{\epsilon_1})\theta^{q_1}(u_1)\theta^{p_2}(a_{\mu_2}^{\epsilon_2})\theta^{q_2}(u_2)\cdots\theta^{p_{n}}(a_{\mu_n}^{\epsilon_n})\theta^{q_n}(u_n)
\end{equation}
where $q_i=s-(\epsilon_1+\cdots+\epsilon_{i})$ and where
$$p_i=\begin{cases}
s-(\epsilon_1+\cdots+\epsilon_{i-1}) & \textrm{ if $\epsilon_{i}=+1$}\\
s-(\epsilon_1+\cdots+\epsilon_{i-1})+1 & \textrm{ if $\epsilon_{i}=-1$.}
\end{cases}$$
Each $\theta^{q_i}(u_i)\in\langle b_1,\ldots,b_l\rangle$ and $\theta^{p_i}(a_{\mu_i}^{\epsilon_i})\in\langle b_1,\ldots,b_l\rangle a_{\xi_i}^{\epsilon_i}\langle b_1,\ldots,b_l\rangle$ where $\xi_i=\mu_i+p_i$ (mod 2).  The result follows unless some $a_0^{\pm1}a_0^{\mp1}$ or $a_1^{\pm1}a_1^{\mp1}$ cancels in the reduction of the right side of equation (\ref{whg hyp special case}).
Such a cancellation would have to occur within $\theta^{p_i}(a_{\mu_i}^{\epsilon_i})\theta^{q_i}(u_i)\theta^{p_{i+1}}(a_{\mu_{i+1}}^{\epsilon_{i+1}})$ for some $i$.  But then
$$\theta^{p_i}(a_{\mu_i}^{\epsilon_i})\theta^{q_i}(u_i)\theta^{p_{i+1}}(a_{\mu_{i+1}}^{\epsilon_{i+1}})\in\langle b_1,\ldots,b_l\rangle.$$
Applying $\theta^{-p_i}$, we get
$$a_{\mu_i}^{\epsilon_i}\theta^{q_i-p_i}(u_i)\theta^{p_{i+1}-p_i}(a_{\mu_{i+1}}^{\epsilon_{i+1}})\in\langle b_1,\ldots,b_l\rangle.$$
For $a_0^{\pm1}a_0^{\mp1}$ or $a_1^{\pm1}a_1^{\mp1}$ to cancel, we would  have to have $\epsilon_{i+1}=-\epsilon_i$.  But then   $p_{i+1}=p_i$ (check the two cases $\epsilon_i=\pm1$), and so
$$a_{\mu_i}^{\epsilon_i} \theta^{q_i-p_i}(u_i) a_{\mu_{i+1}}^{-\epsilon_{i}}\in\langle b_1,\ldots,b_l\rangle.$$
But then $\theta^{q_i-p_i}(u_i)=1$   and $\mu_i=\mu_{i+1}$, because $\langle a_0, a_1, b_1, \ldots, b_l \rangle$ is free on the given generators. This would contradict the fact that $u_i\neq1$.
Thus no such cancellation occurs, and the lemma is proved.
\end{proof}

Next we give a technical lemma comparing the location of the final $(a_rt)^{\pm 1}$ in a $\Lambda_r$-word $w$ to the location of the corresponding $a_r^{\pm 1}$ in its normal form $\wh{w}t^m$, and moreover in  $\theta^n(\wh{w})$. A point of terminology: the `R' in the word `WORD' \emph{occurs 3 letters in}.

We use $\abs{u}_F$ to denote the length of a word $u$ on $a_0, \ldots, a_r, b_1, \ldots, b_l$.  And we use $\abs{u}_{\Lambda_r}$ to denote the length of a $\Lambda_r$-word---that is, length as a word on  $(a_0t), \ldots, (a_rt), b_1, \ldots, b_l$, not as a word on  $a_0, \ldots, a_r, b_1, \ldots, b_l$.

\begin{lemma} \label{lem: N exists} 
For all integers $ A, B \geq 0$ and $r$ with $k\geq r>1$, there exists $N$ such that if $|n|\leq B$ and $w = u (a_r t)^{\pm 1}$ is a reduced $\Lambda_r$-word  with $|u|_{\Lambda_r} \geq N$,   then the final
$a_r^{\pm 1}$ in $\theta^n(\wh{w})$  occurs at least $A$ symbols in.

Similarly, such an $N$ exists for the final $a_0^{\pm 1}$ or $a_1^{\pm 1}$ in  $\theta^n(\wh{w})$  when $r=1$ and $w$ is  $u(a_0t)^{\pm 1}$ or $u(a_1t)^{\pm 1}$.
\end{lemma}

\begin{proof} 
We may suppose $n$ is a fixed integer such that $\abs{n} \leq B$. 

Assume first that  $r>1$. In the manner explained in the comment following
Lemma~\ref{normal form of lambda word}, the   $(a_rt)^{\pm 1}$ in $w$ correspond to the  $a_r^{\pm 1}$ in $\theta^n(\wh{w})$.

First  we will address the case where $w = u (a_r t)^{-1}$.  Since the final letter of $\theta^{\pm 1}(a_r^{-1})$ is  $a_r^{-1}$, the final $a_r^{-1}$ in  $\theta^n(\wh{w})$ is in fact  the final letter of $\theta^n(\wh{w})$.  So we are seeking to prove  that  $| \theta^n(\wh{w}) |_{F} \geq A$ when $N$ is sufficiently large.   
 Since $\theta$ is an automorphism, only finitely many $v\in F$ satisfy $|\theta^n(v)|_F<A$.
By Lemma~\ref{uniqueness from normal form hyperbolic}, each such $v$ equals
$\wh{w}$ for at most one $\Lambda_r$-word $w$, so the lemma is proved by taking $N$
sufficiently large to avoid these finitely many $w$.

Next we address the case where $w = u (a_r t)$.  The normal forms of $u$ and $w$ are $\wh{w}t^{j+1}$  and $\wh{u}t^j$ for some $j$.  They  are related in that    $$\wh{w} t^{j+1} \ = \ w \ = \ u (a_rt) \ = \ \wh{u}t^j \, a_rt \ = \ \wh{u} \, \theta^{-j}(a_r) \, t^{j+1}.$$ 
Thus $\wh{w}=\wh{u}\theta^{-j}(a_r)$. 

Since  the first letter of $\theta^{n-j}(a_r)$ is $a_r$,  Lemma \ref{normal form of lambda word} says there is no cancellation between $\theta^n(\wh{u})$ and $\theta^{n-j}(a_r)$ in  $\theta^n(\wh{w}) = \theta^n(\wh{u})\cdot\theta^{n-j}(a_r)$.  (Otherwise there would be too few instances of $a_r$ in $\theta^n(\wh{w})$.) Thus the final $a_r$ in $\theta^n(\wh{w})$ occurs $|\theta^n(\wh{u})|_{F} +1$ letters in. So we are now seeking to prove that $|\theta^n(\wh{u})|_{F} +1 \geq A$  when $N$ is sufficiently large,  and this can be  handled as in the previous case.  
This completes the proof when $r > 1$.

Next, we do the case $r=1$.   Suppose $w=u(a_\mu t)^{\pm1}$ is reduced with $u=\wh{u}t^j$ and $\mu\in\{0,1\}$.  This time we have $\wh{w}=\wh{u}\theta^{-J}(a_\mu^{\pm1})$ where $|J-j|\leq1$.
This time there may be some cancellation between $\theta^n(\wh{u})$ and $\theta^{n-J}(a_\mu^{\pm1})$ in $\theta^n(\wh{w})=\theta^n(\wh{u})\cdot\theta^{n-J}(a_\mu^{\pm1})$, but the cancellation is restricted (by Lemma \ref{normal form of lambda word special case}) to only symbols from $\{b_1^{\pm 1},\ldots,b_l^{\pm 1}\}$.  We have two cases: $|J|\geq A$ or $|J|< A$. If $|J|\geq A$, then $|j|\geq A-1$, so Lemma \ref{normal form of lambda word special case} tells us that $\theta^n(\wh{u})$ contains at least $A-1$ symbols from among $\{a_0^{\pm1},a_1^{\pm1}\}$. Thus the final $a_0^{\pm1}$ or $a_1^{\pm1}$ of $\theta^n(\wh{w})$ occurs at least $A$ symbols in, as desired.  On the other hand, if $|J|<A$, then $|n-J|<n+A$ by the triangle inequality.  This gives a bound on $|\theta^{n-J}(a_\mu^{\pm1})|_F$ and hence on the amount of cancellation that can occur in $\theta^n(\wh{w})=\theta^n(\wh{u})\cdot\theta^{n-J}(a_\mu^{\pm1})$.
The lemma now follows by the same argument as in the $r>1$ case.
\end{proof}

Let $C(F)$ and $C(\Gamma_k)$ denote the Cayley graphs of $F$ and $\Gamma_k$ with respect to  $a_0$, $\ldots$, $a_k$, $b_1$, $\ldots$, $b_l$ and  $a_0$, $\ldots$, $a_k$, $b_1$, $\ldots$, $b_l$, $t$, respectively.  Let $B_{F}(e,R)$ denote the open ball  of radius $R$ about $e$  in $C(F)$.   Write $[x,y]_{F}$ or $[x,y]_{\Gamma_k}$ for a geodesic between $x$ and $y$ in $C(F)$ or $C(\Gamma_k)$, respectively.  In the case of $C(F)$, which is a tree, geodesics between any given pair of points are unique.  Let $d_F$ and $d_{\Gamma_k}$ be the associated metrics.

The \emph{shadow} of the suffix $\beta$ of a reduced $\Lambda_k$-word $\alpha\beta$ is the set of all geodesic segments $[\wh{\alpha\cdot\beta(i)},\wh{\alpha\cdot\beta(i+1)}]_F$ where $\beta(i)$ denotes the length-$i$ prefix of $\beta$ and $0\leq i<|\beta|_{\Lambda_k}$. 
 
\begin{lemma}
\label{lemma:farshadow}
For all $K>0$, there exist integers $C, R>0$ such that if   $\alpha\beta$ is a reduced $\Lambda_k$-word and $| \alpha |_{\Lambda_k} \geq C$ and the shadow of $\beta$ is outside    $B_F(e,R)$, then every $[\alpha, \alpha \beta]_{\Gamma_k}$ satisfies $d_{\Gamma_k}([\alpha,\alpha\beta]_{\Gamma_k},e) \geq K$.   
\end{lemma}

\begin{proof}
As $\Gamma_k$ is hyperbolic, there is some $\delta >0$ such that every geodesic triangle in $C(\Gamma_k)$ is $\delta$-slim. 

Any geodesic segment $[\alpha,\alpha\beta]_{\Gamma_k}$ is in a $2\delta$-neighborhood of any piecewise-geodesic path $[\alpha,\wh{\alpha}]_{\Gamma_k}\cup[\wh{\alpha},\wh{\alpha\beta}]_{\Gamma_k}\cup[\wh{\alpha\beta},\alpha\beta]_{\Gamma_k}$  in $C(\Gamma_k)$.  So it suffices to have $[\alpha,\wh{\alpha}]_{\Gamma_k}$, $[\wh{\alpha},\wh{\alpha\beta}]_{\Gamma_k}$, and $[\wh{\alpha\beta},\alpha\beta]_{\Gamma_k}$   stay  at least $K + 2\delta$  away  from the identity element $e$.

To  ensure $d(e, [\alpha,\wh{\alpha}]_{\Gamma_k}) \geq K + 2\delta$, we need  $[\alpha,\wh{\alpha}]_{\Gamma_k}$ to avoid finitely many  elements   of  $ {\Gamma_k}$, say $g_1, \ldots, g_m$.   Since  $\alpha = \wh{\alpha} t^n$ for some $n$, there is  a unique geodesic $[\alpha, \wh{\alpha}]$ joining $\alpha$ to $\wh{\alpha}$ in $C(\Gamma_k)$  and it is a succession of edges all labelled $t$.   So if $g_i$ is on $[\alpha, \wh{\alpha}]$ then $g_i t^{j_i} = \alpha \in \Lambda_k$  for some $j_i$. But then by  Lemma~\ref{uniqueness of exponent hyperbolic}, it suffices for  $\alpha$ not to be one of  at most $m$   elements  of $\Lambda_k$.  So  it suffices to ensure $| \alpha |_{\Lambda_k}$ is sufficiently long.     

Since $| \alpha \beta |_{\Lambda_k} \geq  | \alpha |_{\Lambda_k}$, we find   $d([\wh{\alpha\beta},\alpha\beta]_{\Gamma_k}, e) \geq K + 2 \delta$ also.

Finally, we consider  $[\wh{\alpha},\wh{\alpha\beta}]_{\Gamma_k}$.   
The strategy is to use the existence of the Cannon--Thurston map $\partial F\to\partial\Gamma_k$ (not $\partial\Lambda_k\to\partial\Gamma_k$!) to ensure this geodesic stays far (at least $K+2\delta$) from the identity.
As previously mentioned, the main theorem in \cite{CTnormal} is that Cannon--Thurston maps always exist for infinite hyperbolic normal subgroups of hyperbolic groups.  Thus condition \ref{Mitra M} of Lemma ~\ref{lemma:Mitra} must hold where $X=C(\Gamma_k)$ and $Y=C(F)$ and $f$ is the inclusion map.
Since $C(F)$ is a tree, the geodesic segment $[\wh{\alpha},\wh{\alpha\beta}]_F$ is a subset of the union of the geodesic segments comprising the shadow of $\beta$, which do not intersect $B_F(e,R)$ by assumption.  Thus $[\wh{\alpha},\wh{\alpha\beta}]_F$ is disjoint from $B_F(e,R)$.  So condition \ref{Mitra M} says that choosing $R$ large enough makes $[\wh{\alpha},\wh{\alpha\beta}]_{\Gamma_k}$  arbitrarily far (so at least $K+2\delta$) from $e$, as desired.
\end{proof}

We are now ready to use Corollary~\ref{cor:2point3} to show the Cannon--Thurston map $\partial {\Lambda_k}\to\partial {\Gamma_k}$ exists.

\begin{proof}[Proof of Theorem~\ref{existence thm}]
Fix some $\delta>0$ so that all geodesic triangles in $C(\Gamma_k)$ are $\delta$-slim.

For  integers $A, B \geq 0$ and $r$ with $k\geq r\geq 1$, let $N(r, A, B)$ be the least integer $N$ as per Lemma~\ref{lem: N exists}.  Given an integer $R>0$, recursively define a sequence $N_k(R),N_{k-1}(R),\ldots,N_1(R)$ of positive integers   by: 
$$N_k(R) \ :=  \ N(k,R,0),$$
and for $r=k-1, k-2, \ldots, 1$     
$$N_{r}(R) \ := \ N\left(r,R+\max\left\{    |\wh{w}|_{F}    \left| \,    \Lambda_k\textup{-words }  w  \textup{ with }  |w|_{\Lambda_k}= \sum_{j=r+1}^k N_j(R) \right.  \right\},\sum_{j=r+1}^k N_j(R) \right).$$

Suppose $M''>0$ is given.  Let $K=M''+(2\delta+1)k$.  Let $C,R>0$ be obtained from $K$ as per Lemma~\ref{lemma:farshadow}.
Recall the map $\Phi$ from the hyperbolic hydra group $\Gamma_k$ to the hydra group $G_k$ defined by  $a_i \mapsto a_{\max \{1,i\}}$, $b_j \mapsto 1$ and $t \mapsto t$ for all $i,j$.
Let $L$ be the maximum of $|\theta^{-n}(s)|_F$ ranging over all  $s \in \{b_1,\ldots,b_l\}$ and all $n \in \Z$ for which there exists  $\wh{u}\in F$ such that $\wh{u}t^n\in\Lambda_k$ and the reduced word representing $\Phi(\wh{u})$ in $F(a_1, \ldots, a_k)$ has length  less than $R$.  
(There are only finitely many such $n$ since $\wh{u}t^n\in\Lambda_k$ implies $\Phi(\wh{u}t^n) = \Phi(\wh{u}) t^n \in H_k$, and for any   $x \in F(a_1, \ldots, a_k)$, there is at most one $m \in \Z$ such that $xt^m \in H_k$ 
by Lemma~\ref{uniqueness of exponent}.) 
Choose $C'>C$ so that every $\Lambda_k$-word $u$ of length $\abs{u}_{\Lambda_k}  \geq C'$ has free-by-cyclic normal form $\wh{u}t^n$ with $\abs{\wh{u}}_F \geq R+(L/2)$.

Suppose $\alpha \beta$ is a reduced $\Lambda_k$-word such that $|\alpha|_{\Lambda_k} \geq N :=  C' + \sum_{r=1}^k N_{r}(R)$. 
Express $\alpha$ as $w_k\cdots w_0$ where $|w_{r}|_{\Lambda_k}=N_{r}(R)$ for   $1 \leq r \leq k$.
In particular, $|\alpha|_{\Lambda_k}  \geq |w_0|_{\Lambda_k} \geq C'>C$.
Let $\beta_r$ denote the longest prefix of $\beta$ in $\Lambda_r$ and let $\gamma_r$ denote (any) geodesic $[\alpha,\alpha\beta_r]_{\Gamma_k}$.   In particular, $\gamma_k$ is an arbitrary   geodesic $[\alpha,\alpha\beta]_{\Gamma_k}$.
We will show that $\gamma_k$ lies at least a distance $M''$ from $e$ in $C(\Gamma_k)$.
Corollary~\ref{cor:2point3} will then complete the proof.

Suppose, for a contradiction, that $d_{\Gamma_k}(\gamma_k, e)<M''$.   Let $\wh{\alpha}t^n$ be the free-by-cyclic normal form of $\alpha$ in $\Gamma_k$.  

We claim that the shadow of the suffix $\beta_0$ of $\alpha \beta_0$ does not intersect $B_F(e,R)$.
The endpoints of the geodesic segments in $F$ comprising this shadow are all of the form $\wh{\alpha}\theta^{-n}(x)$ for various $x\in F(b_1,\ldots, b_l)$.  There are two cases to consider: the length of the reduced word in $F(a_1, \ldots, a_k)$ representing $\Phi(\wh{\alpha})$ is at least $R$ and is less than $R$. 
In the former case,  because $\wh{\alpha}$ contains at least $R$ letters $a_i^{\pm 1}$ ($0 \leq i \leq k$), the  closest approach of any such geodesic to $e$ (i.e.\ the Gromov product of its endpoints) is at least $R$. 
In the latter case, $L$ is an upper bound for the length of the constituent geodesics in the shadow of $\beta_0$, and so, by  definition of $C'$, the shadow of $\beta_0$ does not intersect $B_F(e,R)$.
In either case, the shadow stays outside of $B_F(e,R)$.

On the other hand, the following claim, in the case $r=0$, 
shows that $$d_{\Gamma_k}([\alpha,\alpha\beta_0]_{\Gamma_k},e) \ <  \ M''+(2\delta+1)k \ = \ K,$$ so Lemma~\ref{lemma:farshadow} implies the shadow of $\beta_0$  intersects $B_F(e,R)$.
This contradiction will prove the theorem.

\emph{\textbf{Claim.}}  For $r = k, k-1, \ldots, 1, 0$,
\vspace*{-6pt}
\begin{enumerate}
 \item[($\textit{i}_r$).] $d_{\Gamma_k}(\gamma_r,e)<M''+(2\delta+1)(k-r)$, and
 \item[($\textit{ii}_{r}$).] $w_rw_{r-1}\cdots w_0\in\Lambda_r$.
\end{enumerate}
We prove this claim using   downward induction on $r$.  

The base case $r=k$ is straightforward:   $\gamma_k = [\alpha, \alpha \beta]_{\Gamma_k}$   and $w_k\cdots w_0=\alpha\in\Lambda_k$ by definition, and   $d_{\Gamma_k}(\gamma_k, e) < M''$ by hypothesis.  

Now we prove  that ($\textit{i}_{r+1}$) and ($\textit{ii}_{r+1}$) implies ($\textit{i}_r$) and ($\textit{ii}_r$) for $r = k-1, \ldots, 1, 0$. (In the case $r=0$, we must interpret each instance of ``$(a_{r+1}t)^{\pm 1}$'' in the following to mean ``$(a_0 t)^{\pm1}$ or $(a_1 t)^{\pm1}$'' and each occurrence of ``$a_{r+1}^{\pm 1}$'' to mean ``$a_{0}^{\pm 1}$ or $a_{1}^{\pm 1}$''.)

We will make repeated use of the following lemma.

\begin{lemma} 
\label{lemma:farshadow2}
Suppose $w_k$,  \ldots, $w_{r+1}$ are as defined earlier, and  $w_k\cdots w_{r+1}x(a_{r+1}t)^{\pm
1}yz$ is a reduced $\Lambda_{k}$-word in which the subwords  $x$, $y$ and $z$ are $\Lambda_{r+1}$-words and $|w_k\cdots w_{r+1}x(a_{r+1}t)^{\pm
1}y|_{\Lambda_k}\geq C$.  If $\gamma$ is any geodesic in $C(\Gamma_k)$ from $w_k\cdots w_{r+1}x(a_{r+1}t)^{\pm
1}y$ to $w_k\cdots w_{r+1}x(a_{r+1}t)^{\pm
1}yz$, then $d_{\Gamma_k}(\gamma, e)  \geq K$.   
\end{lemma}

To prove this lemma we consider the free-by-cyclic normal form $\widehat{ w_k\cdots w_{r+2} }t^{n_r}$ for  $w_k\cdots w_{r+2}$ in $\Gamma_k$.
Now, $\abs{w_k \cdots w_{r+2} }_{\Lambda_k} = \sum_{j=r+2}^k N_j(R)$ by construction, so $\abs{n_r} \leq \sum_{j=r+2}^k N_j(R)$.
So, by definition, $$N_{r+1}(R) \ \geq \  N(r+1,R+|\widehat{w_k\cdots w_{r+2}}|_{F},|n_r|).$$
By hypothesis  $(\textit{ii}_{r+1})$, we see $w_{r+1}x(a_{r+1}t)^{\pm
1}  \in\Lambda_{r+1}$, and $\abs{w_{r+1}}=N_{r+1}(R)$ by construction, so
the final $a_{r+1}^{\pm 1}$ in $\theta^{-n_r}(\widehat{w_{r+1}x(a_{r+1}t)^{\pm
1}})$ is at least $R+|\widehat{w_k\cdots w_{r+2}}|_{F}$ letters in.
Therefore, the final $a_{r+1}^{\pm 1}$ in
$\widehat{w_k\cdots w_{r+2}} \ \theta^{-n_r}( \wh{w_{r+1}x(a_{r+1}t)^{\pm 1}} )$
is at least $(R+|\widehat{w_k\cdots w_{r+2}}|_{F})-|\widehat{w_k\cdots w_{r+2}}|_{F}=R$ letters in.
Since $y$ and $z$ are $\Lambda_{r+1}$-words and $w_k\cdots w_{r+1}x(a_{r+1}t)^{\pm
1}yz$ is reduced, this implies the shadow of $z$ cannot intersect $B_F(e,R)$.
So, by Lemma~\ref{lemma:farshadow},   $d_{\Gamma_k}(\gamma,e)\geq K$, completing the proof of the lemma.

Returning to the proof of the claim, we will consider two cases: $\beta_r=\beta_{r+1}$ and $\beta_r\neq\beta_{r+1}$.
In the former case, we may assume $\gamma_r=\gamma_{r+1}$.
So ($\textit{i}_r$) follows immediately from ($\textit{i}_{r+1}$).
Since ($\textit{i}_r$) implies $d_{\Gamma_k}(\gamma_r,e) < K$,
 ($\textit{ii}_{r+1}$) and Lemma~\ref{lemma:farshadow2} with $z=\beta_r$ shows that $w_r\cdots w_0$
cannot be expressed as  $x(a_{r+1}t)^{\pm 1}y$.   So $w_r\cdots
w_0\in  \Lambda_{r}$ and we have ($\textit{ii}_r$).

Next, assume  $\beta_r\neq\beta_{r+1}$.  The (reduced) word $\beta_{r+1}$ can be expressed as $\beta_r(a_{r+1}t)^{\pm 1}\beta_r'$ for some $\beta_r'\in\Lambda_{r+1}$.
Let $\rho_{r}$ be the geodesic segment in $C(\Gamma_k)$ labelled $(a_{r+1}t)^{\pm 1}$ connecting $\alpha\beta_r$ to $\alpha\beta_r(a_{r+1}t)^{\pm 1}$.
Let $\gamma_r'=[\alpha\beta_r(a_{r+1}t)^{\pm 1},\alpha\beta_{r+1}]_{\Gamma_k}$.  Then, by Lemma~\ref{lemma:farshadow2} with $x=w_r\cdots w_0\beta_r$,  $y$ the empty word, and $z=\beta_r'$,
\begin{equation}
 \label{eq:gamma prime}
 d_{\Gamma_k}(\gamma_r',e) \ \geq \ K.
\end{equation}

If $(a_{r+1}t)^{\pm 1}$ occurs in $w_r\cdots w_0$, then Lemma~\ref{lemma:farshadow2} with $x(a_{r+1}t)^{\pm 1}y=w_r\cdots w_0$ and $z=\beta_r$  shows that 
\begin{equation}
 \label{eq:gamma itself}
 d_{\Gamma_k}(\gamma_r,e) \ \geq \ K.
\end{equation}
 
By the slim-triangles condition for $C(\Gamma_k)$, 
$\gamma_{r+1}$ is contained in  the $2\delta$-neighborhood of  $\gamma_r\cup \rho_{r} \cup \gamma_r'$ and hence in the $(2\delta+1)$-neighborhood  of $\gamma_r\cup \gamma_r'$. 
So  
\begin{align*}
\min\{d(e,\gamma_r),d(e, \gamma_r')\}  \ & \leq \  d(e,  \gamma_{r+1})+(2\delta+1) \\ 
& < \ M''+  (2\delta+1)(k-(r+1)) +(2\delta+1) \\
&  =  \ M''+(2\delta+1)(k-r), 
\end{align*}
the second inequality coming from ($\textit{i}_{r+1}$).  

But by (\ref{eq:gamma prime}), $d(e,\gamma_r')\geq K\geq M''+(2\delta+1)(k-r)$. So $\min\{d(e,\gamma_r),d(e, \gamma_r')\}=d(e,\gamma_r)$ and ($\textit{i}_r$) follows. 

Moreover, (\ref{eq:gamma itself}) cannot be true since it contradicts ($\textit{i}_r$), so $(a_{r+1}t)^{\pm 1}$ does not occur in $w_r\cdots w_0$ and ($\textit{ii}_r$) follows.
This completes the induction step of the claim, and thus proves the theorem by contradiction.
\end{proof}

\section{Wildness of Cannon--Thurston maps}  \label{Wild continuity} \label{Wildness}
\begin{proof}[Proof of  Theorem~\ref{Distortion and Wildness}.]
  We have that $\Gamma$ and $\Lambda$ are $(\delta_{\Gamma})$- and $(\delta_{\Lambda})$-hyperbolic, respectively, for some  $\delta_{\Gamma}, \delta_{\Lambda} > 0$.  Let $\imath: \Lambda \to \Gamma$ denote the inclusion map and  $\ihat: \partial \Lambda \to \partial \Gamma$ denote the Cannon--Thurston map.
 
Since $\Lambda$ is non-elementary,
   $|\ihat(\partial \Lambda)|=\infty$ by \cite[Thm. 12.2(1)]{KB}.  We may thus choose 
$p_1,p_2,p_3\in\partial \Lambda$
with $\ihat p_1,\ihat p_2, \ihat p_3\in\partial \Gamma$ distinct. Let 
$C=2\delta_{\Lambda}+\max\{(p_i \cdot p_j)_e^{\Lambda} \mid 1\leq i<j\leq 3\}$.

By definition of the distortion function (see Section~\ref{intro}), we can take a sequence $h_n\in \Lambda$ with 
$d_{\Gamma}(e,h_n)\leq n$ and $d_{\Lambda}(e,h_n)=\Dist_{\Lambda}^{\Gamma} (n)$.  By Lemma~\ref{BH1 433 Lemma}\ref{BH1 433 Lemma 1}
\begin{align*}
(p_i\cdot p_j)_e^{\Lambda}& \ = \ (h_np_i\cdot h_np_j)_{h_n}^{\Lambda} \ \geq \  \min \lbrace (e\cdot h_np_i)_{h_n}^{\Lambda},(e\cdot h_np_j)_{h_n}^{\Lambda}\rbrace-2\delta_{\Lambda},
\end{align*}
so we can choose $i=i(n)$ and $j=j(n)\in\{1,2,3\}$ with $i\neq j$ such that
\begin{align*}
(e\cdot h_np_i)_{h_n}^{\Lambda}, (e\cdot h_np_j)_{h_n}^{\Lambda} \ \leq \  C.
\end{align*}

Combined with Lemma~\ref{BH1 433 Lemma}\ref{BH1 433 Lemma 2} this gives that for $k = i,j$, 
\begin{align*}
(h_n\cdot h_np_k)_e^{\Lambda} \ \geq \  d_{\Lambda}(e,h_n)-(e\cdot h_np_k)_{h_n}^{\Lambda}-\delta_{\Lambda} \ \geq \ \Dist_{\Lambda}^{\Gamma}(n)-C-\delta_{\Lambda}.
\end{align*}
So, by Lemma \ref{BH1 433 Lemma}\ref{BH1 433 Lemma 1}, $$(h_np_i\cdot h_np_j)_e^{\Lambda} \ \geq \ \min\{(h_n\cdot h_np_i)_e^{\Lambda},(h_n\cdot h_np_j)_e^{\Lambda}\}-2\delta_{\Lambda} \ \geq \ \Dist_{\Lambda}^{\Gamma}(n)-C-3\delta_{\Lambda}.$$

Writing $\beta:=k_2   r^{C+3\delta_{\Lambda}}$,  where $k_2$ is as per \eqref{visual eq} in Section~\ref{hyperbolic review} applied to $\partial \Lambda$, we get
\begin{equation}
\label{Distortion and Wildness 1}
d_{\partial \Lambda}(h_np_i,h_np_j) \ \leq \  k_2 r^{ -(h_np_i \cdot h_np_j)_e^{\Lambda} }  \ \leq \ \frac{\beta}{r^{\Dist_{\Lambda}^{\Gamma}(n)}}.
\end{equation}

On the other hand, using Lemma \ref{BH1 433 Lemma}\ref{BH1 433 Lemma 3} for the first inequality,
\begin{align*}
(\ihat(h_np_i)\cdot \ihat(h_np_j))_e^{\Gamma} \ &  = \ (h_n\ihat p_i\cdot h_n\ihat p_j)_e^{\Gamma}\\
&\leq \  d_{\Gamma}(e,[h_n\ihat p_i,h_n\ihat p_j]_{\Gamma})+ {8\delta_{\Gamma}}\\
&\leq \ d_{\Gamma}(e,h_n)+d_{\Gamma}(h_n,[h_n\ihat p_i,h_n\ihat p_j]_{\Gamma})+ {8\delta_{\Gamma}}\\
&\leq \ n+d_{\Gamma}(e,[\ihat p_i,\ihat p_j]_{\Gamma})+ {8\delta_{\Gamma}}.
\end{align*}

Writing $\alpha   :=   k_1/s^{{8\delta_{\Gamma}}+\max\{d_{\Gamma}(e,[\ihat p_i,\ihat p_j]_{\Gamma}) \mid 1\leq i<j\leq 3\}}$,   where $k_1$ is as per \eqref{visual eq} in Section~\ref{hyperbolic review} applied to $\partial {\Gamma}$,  we then get
\begin{equation}
\label{Distortion and Wildness 2}
d_{\partial \Gamma}(\ihat(h_np_i),\ihat(h_np_j))  \ \geq \  k_1 s^{ - (\ihat(h_np_i),\ihat(h_np_j))_e^{\Gamma}  }   \ \geq \ \frac{\alpha}{s^n}.
\end{equation}

Combining   (\ref{Distortion and Wildness 1}) and (\ref{Distortion and Wildness 2}) yields the   inequality claimed in Theorem~\ref{Distortion and Wildness}. 
\end{proof}

Corollary~\ref{wildness thm} will follow from the following proposition.

\begin{prop}  \label{wildness prop}
For all $k \geq2$, the modulus of continuity $\varepsilon(\delta)$ of the Cannon--Thurston  map $\partial{\Lambda_k}\to\partial{\Gamma_k}$ for hyperbolic hydra  has the property that there exist $C_0, C_1 >0$ and $C_2>1$  such that for all ${\eta} \in (0, C_0)$,  
 \begin{equation} \label{modulus prop eqn}
  \varepsilon \left(  \frac{1}{ C_2^{A_k \left( \lfloor C_1 \log (1/{\eta})\rfloor \right)}} \right)   \ \geq \  {\eta}.
  \end{equation}
\end{prop}

\begin{proof} 
For convenience, extend the domains of the functions $A_k : \mathbb{N} \to \mathbb{N}$ and $\Dist^{\Gamma_k}_{\Lambda_k} :\mathbb{N} \to \mathbb{N}$ to  $[1,\infty)$ by declaring the functions to be constant on the half-open intervals $[n, n+1)$.

From  Theorem~\ref{Distortion and Wildness} we have   \[ \varepsilon\left(\frac{\beta}{r^{\Dist^{\Gamma_k}_{\Lambda_k}(n)}}\right) \ \geq \ \frac{\alpha}{s^n} \] for all real $n \geq 0$ and some constants $\alpha, \beta > 0 $ and $r, s >1$.
Thus  for all ${\eta} \leq \alpha $,
\begin{equation}
   \varepsilon\left(\frac{1}{\exp{\left(\log(r)\Dist^{\Gamma_k}_{\Lambda_k}(\log_s (\alpha / {\eta}) )-\log(\beta)\right)}}\right) \ \geq \ {\eta}. 
 \label{varep bound}
\end{equation}

As   $\Dist^{\Gamma_k}_{\Lambda_k} \succeq A_k$ by  \cite{HypHydra}, there exists $C >0$ such that $A_k(n) \leq C  \Dist^{\Gamma_k}_{\Lambda_k} ( Cn +C ) + Cn+C$ for all real $n$, and therefore
$$\Dist^{\Gamma_k}_{\Lambda_k} (N) \ \geq \  \frac{1}{C} A_k \left(  \frac{N-C}{C}  \right) - \frac{N}{C}$$
for $N \geq 2C$.   So 
\begin{equation*}
\log(r)\Dist^{\Gamma_k}_{\Lambda_k}(\log_s (\alpha / {\eta}) ) -\log(\beta) \ \geq \ K_1 A_k ( K_2 \log (1/ {\eta}) - K_3) - K_4 \log (1/ {\eta}) - K_5 
\end{equation*}
 for all ${\eta} \in (0, K_0)$, for suitable constants $K_0, \ldots, K_5>0$.  By shrinking $K_0$ if necessary, we can make $\log(1/{\eta})$ arbitrarily large, so that we may absorb the constant $K_3$ into $K_2$.  Moreover, $A_k$ grows faster than a linear function as $k \geq 2$, so (by further shrinking $K_0$) the constants $K_4$ and $K_5$ can be absorbed into $K_1$.  Thus
\begin{equation}
\log(r)\Dist^{\Gamma_k}_{\Lambda_k}(\log_s (\alpha / {\eta}) ) -\log(\beta) \ \geq \ K_1 A_k ( K_2 \log (1/ {\eta})). 
\label{dist bound}
\end{equation}

So combining \eqref{varep bound} and \eqref{dist bound} and setting $C_0=K_0$, $C_1=K_2$, and $C_2=e^{K_1}$, we have the result claimed.
\end{proof}

Plugging $\eta=1/n$ into inequality (\ref{modulus prop eqn}) yields 
$$\varepsilon \left(  \frac{1}{ C_2^{A_k \left( \lfloor C_1 \log (n)\rfloor \right)}} \right)   \ \geq \  {\frac1n}$$
for all sufficiently large $n$.
So Corollary~\ref{wildness thm}, which asserts $$\varepsilon\left(\frac{1}{A_{k-1}(n)}\right)  \ \geq \  \frac{1}{n}$$ for all sufficiently large $n$, follows from the fact that
$$C_2^{A_k \left( \lfloor C_1 \log (n)\rfloor \right)}\gg A_k \left( \lfloor C_1 \log (n)\rfloor \right)
=A_{k-1}\bigg(A_k \left( \lfloor C_1 \log (n)\rfloor -1\right)\bigg)
\ge A_{k-1}\bigg(A_3 \left( \lfloor C_1 \log (n)\rfloor -1\right)\bigg)
\gg A_{k-1}(n).$$
For the last $\gg$ inequality, recall that
$A_3(m) =     \mbox{$2$\raisebox{4pt}{$2$} \!\! \raisebox{10pt}{\reflectbox{$\ddots$}}}   \raisebox{17pt}{$2$}   \raisebox{7pt}{$ \left. \rule{0mm}{16pt} \right\}  $} \! \raisebox{6pt}{$m$}$.

\bibliography{hydrabib}
\end{document}